%% file: main.tex
\documentclass[12pt]{article}

\usepackage[utf8]{inputenc}
\usepackage[T1]{fontenc}

\usepackage{caption}
\usepackage{subcaption}

\usepackage[margin=2.5cm]{geometry}
\usepackage{amsfonts}
\usepackage{amsmath}
\usepackage{amssymb}
\usepackage{amsthm}
\usepackage{enumitem}
\usepackage{xcolor}
\usepackage[hyperfootnotes=false]{hyperref}
\usepackage{graphicx}
\usepackage{booktabs}
\usepackage{lipsum}

\usepackage{stmaryrd}

\include{macros}

\begin{document}

\title{
Exact targeting of Gibbs distributions using velocity-jump processes.
}
\author{
P.~Monmarch{\'e}\footnote{LJLL $\&$ LCT -- Laboratoire Jacques-Louis Lions and Laboratoire de Chimie Th\'eorique, Sorbonne Universit\'e }, %
M.~Rousset\footnote{Inria $\&$ IRMAR -- Institut de Recherche en Math\'ematiques de Rennes, Univ Rennes} 
and
P.A.~Zitt\footnote{
LAMA, Univ Gustave Eiffel, Univ Paris Est Creteil, CNRS, F-77454 Marne-la-Vallée, France
}
}

\date{August 2020}
\maketitle

\begin{abstract}
This work introduces and studies a new family of velocity jump Markov processes directly amenable to exact simulation with the following two properties: i) trajectories converge in law, when a time-step parameter vanishes, towards a given Langevin or Hamiltonian dynamics; ii) the stationary distribution of the process is always exactly given by the product of a Gaussian (for velocities) by any target log-density. 
The simulation itself, in addition to the computability of the gradient 
of the log-density, depends on the knowledge of appropriate explicit upper bounds on lower order derivatives of this log-density.
The process does not exhibit any velocity reflections (jumps maximum size can be controlled) and is suitable for the 'factorization method'. We provide rigorous mathematical proofs of the  convergence towards Hamiltonian/Langevin dynamics when the time step vanishes, and
of the exponentially fast convergence towards the target distribution when a suitable noise on velocities is present. Numerical implementation is detailed and illustrated.
\end{abstract}

\tableofcontents

%

\section{Introduction}

A kinetic process is a Markov process $(X_t,V_t)_{t\geqslant 0}$, where $X_t \in \R^d$ and $V_t\in\R^d$ are respectively called the position and velocity of the process, such that $X_t = X_0 + \int_0^t V_s \d s$ for all $t\geqslant 0$. In addition to modelling a variety of phenomena, these processes can be used as time continuous Markov Chain Monte Carlo algorithms. In this case, given a target probability distribution $\nu$ on $\R^d$, the idea is to construct a kinetic process 
that is ergodic with respect to some probability measure $\pi$ on $\R^{2d}$ whose first marginal is the target distribution $\nu$. This program generalizes the usual construction of a $\nu$-ergodic process $(X_t)_{t\geqslant 0}$ on $\R^d$. When this ergodicity holds,  for observables $f$ that only depend on the position, the empirical estimation $t^{-1}\int_0^t f(X_s)\d s$ still converges in large times towards $\nu(f)$. This idea traces back to the Molecular Dynamics (MD) of Alder and Wainwright \cite{AlderWainwright}, based on the Hamiltonian dynamics, introduced shortly after the seminal Metropolis algorithm. Beyond physical applications and motivations --- Hamiltonian-based processes simulate the real physical dynamics, an algorithmic motivation is that kinetic processes have a ballistic, rather than diffusive, behaviour: their inertia reduces backtracking, which improves the exploration of the configuration space, by comparison with reversible processes such as Metropolis-Hastings random walk or usual elliptic diffusions.

Langevin diffusion and Hamiltonian Monte-Carlo (HMC) are classical kinetic processes used for sampling purposes. In the last decade, another class of velocity jump samplers has emerged, first obtained as scaling limits of rejection-free lifted Markov chains \cite{PetersdeWith,BierkensRoberts,MicloMonmarche}. In these new samplers, the velocity is piecewise constant and is updated at random times; in particular, the process belongs to the family of piecewise deterministic Markov process (PDMP). The law of these so-called jump (or collision, or event) times is chosen in such a way that the invariant distribution of the process is the target $\pi$. An appealing feature of these processes is that they can be implemented in continuous time, since only the value of the process at its jump time is needed, and no supplementary time discretization is required. In particular, the equilibrium of the process effectively implemented is the correct one, which is usually not the case for discretized diffusions. In HMC-like methods, a Metropolis step is added which corrects for the time discretization; however the introduced rejection requires a velocity reflection which destroys the ballistic dynamics and impairs the efficiency of the algorithm. Another interesting point is that, as detailed in Section~\ref{sec:GeneralSetting} (see also \cite{Monmarche2019Kinetic_walk,Monmarche-MD}), different parts of the log-density of $\nu$ may be treated at different time scales through a factorization of the target measure, thus reducing the overall computational complexity of the algorithm. This property is somewhat analogous 
to the deterministic multi-time-step integration methods \cite{Tuckerman,CarterGibson} but, again, without their statistical bias.

\bigskip

If the user is only interested in computing static quantities, that is, integrals of some observables with respect to $\nu$, then any $\nu$-ergodic process, or $\pi$-ergodic kinetic process with marginal~$\nu$, is theoretically usable, even if some may perform better than others for a finite computational budget. 
 The question is a bit different when the aim is to compute dynamical quantities
 (diffusion constants, escape rates, quasi-stationnary distributions...) for a given, particular kinetic process, typically the Hamiltonian or Langevin dynamics.
 Indeed, though they have the same equilibrium, different kinetic processes may have completely different dynamical properties. For instance,  bouncy-type samplers, HMC, or other Metropolized schemes based on Langevin diffusions \cite{Ottobre2016} all feature occasional reflections of the velocity; such discontinuities never happen in Hamiltonian or Langevin dynamics.

Errors in the computation of dynamical quantities naturally 
occur when the computation is done by discretizing in time the 
continuous time dynamics of interest: a Langevin process discretized with a Verlet-like scheme for example, does not have exactly the same dynamical properties as the reference continuous-time process. In these cases however, there is a parameter, namely the 
discretization time-step $\eps$, which may be tuned to obtain a trade-off between dynamical precision and cost:  smaller $\eps$ lead to a better precision on the dynamical properties, 
at the cost of longer computations ---- simulating a trajectory for a given fixed time $T$
typically requires $T/\eps$ computations of the gradient of the log-density of $\nu$. 
Such a precision/computation cost tradeoff does not currently exist for bouncy-type kinetic samplers.

\bigskip


The main contribution of the present work is the design of a new family of velocity jump processes with two interesting properties. 
Firstly, similarly to discretized Langevin or Hamiltonian schemes, the process does not suffer from regular velocity reflections and moreover converges when a time-step parameter $\eps$ vanishes towards a given Langevin of Hamiltonian dynamics. 
Secondly,  similarly to bouncy-type samplers, it is a kinetic MCMC sampler with exact target distribution and suitable for the factorization method. 

We provide a rigorous mathematical proof of two related properties. The
first one is the convergence in distribution of trajectories of the considered process towards Hamiltonian dynamics, when the 
time-step parameter $\eps$ vanishes. This result relies on classical characterization techniques based on martingale problems. 
The second property we establish is the  exponentially fast convergence of the process time marginal distributions towards the exact target distribution, in an $\L^2$ sense. This result relies on a hypocoercivity analysis based on a Lyapunov function in the form of a well-chosen modified $\L^2$-norm, in the spirit of~\cite{DMS}.

The improvement from Hamiltonian integrators and randomized variants is thus that the static properties are unbiased and, maybe more importantly in this context, the factorization method is still available. The price to pay is the loss of geometric properties as symplecticity.  The improvement from bouncy-type samplers is that the proposed method introduces a time-step parameter $\eps$ that enables to interpolate the former with Hamiltonian/Langevin dynamics.

Finally, we remark that several recent works \cite{Underdamped,DoucetRHMC}  in Bayesian statistics have argued that samplers based on Hamiltonian dynamics or Langevin diffusion have good convergence properties, from the fact the continuous-time limit process has dimension-free convergence rate for smooth and concave potentials, and then controlling the distance between this limit and the effective algorithm. In this context, our family of processes may provide a way to keep the dimension-free convergence rate while suppressing the bias (although the dimension should still intervene in the complexity of the algorithm). As said above, we provide explicit~$\L^2$ convergence rates in the spirit of \cite{DMS} and \cite{Andrieu} under general assumptions.

\bigskip

The article is organized as follows. The general framework of kinetic samplers and velocity jump processes is introduced in Section~\ref{Sec:General}. Section~\ref{Sec:Gaussian} contains the definition of the new family of processes and the proof of convergence toward the Hamiltonian dynamics (Theorem~\ref{th:conv}). Exponential convergence toward equilibrium with explicit rates is established in Section~\ref{Sec:hypoco} through Hypocoercivity arguments (Theorem~\ref{ThmHypoco}).  The effective simulation of the processes is discussed in Section~\ref{Sec:Simulation}, and numerical experiments are provided in  Section~\ref{Sec:numerique}. Finally,  the proof of a general result for the convergence of Markov processes, Theorem~\ref{th:conv_gen}, used in the proof of Theorem~\ref{th:conv}, is postponed to Section~\ref{Sec:Supplement}.

\section{Kinetic samplers}\label{Sec:General}

\subsection{General setting}\label{sec:GeneralSetting}

Let $\nu$ and $\gamma$ be two probability laws on $\R^d$, where $\nu$ admits a density with respect to the Lebesgue measure proportional to $\exp(-U)$, for some function $U\in\calC^1(\R^d)$ --- the log-density.
We are interested in kinetic processes for which the 
Gibbs distribution $\pi = \nu\otimes\gamma$, namely
\begin{equation}\label{eq:gibbs}
  \pi( \d x \d v) \propto \exp\PAR{- U(x)} \d x \gamma(\d v) \,,
\end{equation}
is invariant.

\begin{Rem}[Marginal in the velocities]
  There are several possible choices for $\gamma$. Usual ones are Gaussian distributions and
   the uniform measure on a sphere or
  on a discrete set of velocities. 
\end{Rem}

 Consider a Markov process on $\R^d \times \R^d$ with
--- formal --- generator $\cal L$, decomposed as 
\begin{equation}\label{eq:gen1}
\calL \phi  (x,v) = \calT \phi  (x,v)  + \calF \phi  (x,v)  + \calD \phi  (x,v) 
\end{equation}
for smooth, compactly supported test functions $\phi\in \mathcal C_c^\infty(\R^{2d})$, where:
\begin{itemize}
\item the \emph{transport} part $\calT \phi  (x,v)  = v\cdot \nabla_x \phi(x,v)$ is the free-flight transport operator, and is the only part that acts on the position variable in the sense that $\calD \phi = \calF \phi =0$ if $\phi(x,v)=g(x)$ for some function $g$. In terms of trajectories, this ensures that $X_t = \int_0^t V_s ds$. 
\item the \emph{dissipative} part $\calD$ is a Markov generator that acts on the velocity variables and leaves $\gamma$ invariant: 
\begin{equation}\label{eq:genD}
\forall \phi\in\calC_c^\infty(\R^{2d})\,, \forall x \in \R^d\,,\qquad\int_{\R^d} \calD \phi  (x,v) \gamma(\d v) = 0\,.
\end{equation}
\item the \emph{force} part $\cal F$ acts on velocity variables
and is such that for all $\phi\in\calC_c^\infty(\R^{2d})$,
\begin{equation}\label{eq:genF}
\int_{\R^{2d}} \calF \phi  (x,v) \pi(\d x\d v) =  - \int_{\R^{2d}}  \phi  (x,v) \p{v\cdot \nabla U(x)} \pi(\d x\d v)  \,.
\end{equation}
\end{itemize}
Integrating by parts, we see that this last condition means that $\int \calF \phi \pi = - \int \calT \phi\pi$ for all $\phi\in\calC_c^\infty(\R^{2d})$. As a consequence  \eqref{eq:genD} together with \eqref{eq:genF} imply that $ \pi(\calL \phi ) = 0$ for all $\phi\in\calC_c^\infty(\R^{2d})$. If $\calC_c^\infty(\R^{2d})$ is a core for $\calL$, which is usually true and can be proven through regularization and truncation arguments \cite{DurmusGuillinMonmarche2018_PDMP}, then this implies that $\pi$ is invariant for $\calL$.

Many operators satisfy the requirements for the dissipative part $\calD$; 
let us mention three usual choices:
\begin{itemize}
\item  Friction/Dissipation:
\begin{equation}\label{eq:D1}
\calD \phi  (x,v) = - v\cdot \nabla_v \phi(x,v) + \frac{\sigma^2}{2}\Delta_v \phi(x,v)\,,
\end{equation}
for some $\sigma>0$. In this case $\gamma$ is the centered normal distribution with variance $\sigma^2$, and $\calD$ is the generator 
of an Ornstein-Uhlenbeck process acting on velocities. 
\item Velocity refreshment:
\begin{equation}\label{eq:D2}
\calD \phi  (x,v) = \int_{\R^d} \p{\phi(x,w)-\phi(x,v)} \gamma(\d w)\,.
\end{equation}
In terms of trajectories this corresponds to resampling the velocity 
at rate $1$, according to the equilibrium measure $\gamma$. 
\item Partial refreshment:
\begin{equation}\label{eq:D3}
\calD \phi  (x,v) = \int_{\R^d} \p{\phi(x,p v+\sqrt{1-p^2}w)-\phi(x,v)} \gamma(\d w)\,,
\end{equation}
for some $p\in[0,1)$ if $\gamma$ is a normal distribution. This 
corresponds to changing the velocity at random times, using the transition
kernel of the Ornstein-Uhlenbeck process, and can be seen (up to a rescaling in time) as an  interpolation between the previous two exemples.
\end{itemize}
In general, note that \eqref{eq:genD} implies that for any probability law $\tilde \nu$ on $\R^d$, $\tilde\nu\otimes\gamma$ is invariant for~$\calD$. Moreover, if \eqref{eq:genD} holds, then it also holds for the generator $\calD_2 \phi(x,v) = \eta(x) \calD\phi(x,v)$ for any positive function $\eta$ on $\R^d$. For instance, when $\calD$ models the interaction of the system with an external heat bath, there may be no coupling with the heat bath in the interior of some domain, i.e. $\eta(x) = 0$ for $x$ in the domain, and $\eta(x)>0$  outside. Similarly, if~$\calD_1$ and~$\calD_2$ both satisfy \eqref{eq:genD}, then $\calD_1+\calD_2$ does too.

\bigskip

Let us now discuss in more detail the \emph{force} part~$\calF$. 
The most classical choice here is the deterministic drift operator 
\[
  \calF\phi(x,v) = -\sigma^2\nabla U(x)\cdot\nabla_v \phi(x,v)
\]
 which satisfies~\eqref{eq:genF} if $\gamma$ is the centered normal distribution with variance $\sigma^2$. With this choice, then $\calL$ is the generator of the Hamiltonian dynamics if $\calD=0$, of the Langevin diffusion if $\calD$ is given by \eqref{eq:D1}, or of the HMC if $\calD$ is given by \eqref{eq:D2}. 
 
 The factorization (or splitting) method relies on the following remark. Suppose that $\nabla U(x) = \sum_{i=1}^N \xi_i(x)$ for some vector fields $\xi_i$ on $\R^d$, $i=1..N$, and that we have $N$ operators $\calF_1,\dots,\calF_N$ such that for all $i\in\cco 1,N\ccf$,
 \begin{equation}
 \label{eq:genFi}
\int_{\R^{2d}} \calF_i \phi  (x,v) \pi(\d x\d v) =  - \int_{\R^{2d}}  \phi  (x,v) \p{v\cdot \xi_i(x)} \pi(\d x\d v)  \,.
\end{equation}
Then $\calF = \sum_{i=1}^N \calF_i$ satisfies \eqref{eq:genF}. If $\xi_i = \nabla U_i$ for all $i\in\cco 1,N\ccf$ for some $U_i\in\mathcal C^1(\R^d)$, then the decomposition of $\calF$ is based on the factorization
\[\nu(\d x) \propto \prod_{i=1}^N e^{-U_i(x)} \d x\,.\]
Note that in that case it is not necessary that $\exp(-U_i)$ has finite mass. More generally  $\xi_i$ is not required to be a gradient. For instance, if $(e_i)_{i\in\cco 1,d\ccf}$ is the canonical basis of $\R^d$, then $\xi_i(x) = (\nabla U(x) \cdot e_i) e_i$ gives a decomposition of the forces $\na U$ as a sum of possibly non-gradient forces.

Through such a decomposition, different forces may be treated with different dynamics. For instance, as we will see in Section \ref{Sec:Simulation}, jump mechanisms  are easily simulated if $\na U$ is bounded, or Lipschitz, with a known bound, which is not always the case. On the other hand, drift mechanisms suffer the problem of discretization, and a possibly higher computational cost since the forces have to be computed at each time-step. If $\na U$ can be decomposed in long-range forces which are expansive to compute but easily bounded, and short-range forces which are possibly singular but cheap to compute, then it is natural to treat the first ones with jump processes and the second ones with drift processes \cite{Monmarche-MD}. Similarly, if different forces have different time-scales, then instead of using different time-steps in a numerical integration of a drift mechanism, it is possible to use different jump mechanisms as detailed in Section  \ref{Sec-Multitimestep}.

In the rest of the paper, unless otherwise specified, we will only consider the non-factorized condition \eqref{eq:genF}. Indeed, from an operator $\mathcal F$ that satisfies \eqref{eq:genF} (or more precisely \eqref{eq:F*} below) and whose definition only involves $U$ through $\nabla U$, it is then easy to obtain an operator $\mathcal F_i$ that satisfies \eqref{eq:genFi} by replacing $\nabla U$ by $\xi_i$ everywhere in the definition of $\mathcal F$ (see Section~\ref{Sec-limit-factorize}).

\subsection{Velocity jumps}
Let $\lambda(x,v)$ be a non-negative function, and 
for each $x$, let $k(x,v;\d v')$ be a Markov kernel. We 
denote by $q$ the non-normalized kernel $q(x,v;\d v') = \lambda(x,v) k(x,v; \d v')$. 
From now on, we consider the case where the jumps on the velocity
are given by such a kernel: 
\begin{equation}\label{eq:gen}
\calF \phi  (x,v) 
  = \int_{v' \in \R^d} \p{ \phi(x,v') - \phi(x,v)}  q(x,v; \d v'). 
\end{equation}
In this case, the dynamics of a Markov process with generator $\calT+\calF$
is the following: the $x$ variable evolves deterministically at velocity $v$; the velocity is piecewise constant, and jumps at a rate 
$\lambda(x,v)$ to a new velocity $v'$ sampled according
to $k(x,v;\d v')$. The number of jumps may go to infinity at finite time, unless for instance $\lambda$ is bounded. This kind of process is known as a \emph{velocity jump process}.

In a way that is similar to the classical Metropolis algorithm, the 
jump mechanism $q$ will be constructed by choosing a nice \emph{proposal kernel} $q_0$, and then \emph{modifying it} to take the log-density
$U$ into account, yielding a \emph{corrected kernel}~$q$. We start 
by stating two conditions that our proposal kernel should satisfy. 
\begin{Def}[Conditions for the proposal kernel]
A non-negative kernel $q_0(x,v ; \d v')$ is \emph{reversible}
with respect to $\gamma$ if
\begin{equation}\label{eq:rev}
 q_0(x,v ; \d v')\gamma(\d v) 
 = q_0(x,v' ; \d v) \gamma(\d v') \qquad \forall x \in \R^d. \tag{R}
\end{equation} 

It satisfies the  \emph{average condition} \eqref{eq:av}  if moreover $\int_{v'}1+\abs{v'} q_0(x,v ; \d v') < + \infty$ for all $(x,v) \in \R^{2d}$ and
 \begin{equation}\label{eq:av}
 \nabla  U(x) \cdot  \int_{v' \in \R^d} \frac{1}{2}(v - v') q_0(x,v ; \d v') 
 =    \nabla  U(x) \cdot v,
 \qquad \text{$\d x  \gamma( \d v)$-a.e.}. \tag{A}
\end{equation}
\end{Def}
Note that \eqref{eq:av} may be rewritten in terms of the
intensity $\lambda_0(x,v) = \int q_0(x,v,\d v')$ 
and the normalized kernel $k_0 = q_0/\lambda_0$ 
as
 \begin{equation}\label{eq:av-2}
  \PAR{\int v' k_0(x,v;\d v') }\cdot \nabla  U(x) = \PAR{1 - \frac{2}{\lambda_0(x,v)}} v\cdot \nabla  U(x).
\end{equation}

Let $\psi: \R \to \R_+$ be a measurable function such that 
\begin{equation}
\label{eq:conditionOnPsi}
\psi(s)-\psi(-s) = s, \quad \forall s\in \R.
\end{equation}
The basic choice for $\psi$ is $\psi(s)=(s)_+$, but as remarked in \cite{Andrieu2} there are other possibilites, like $\psi(s) =  a\ln(e^{s/a}+1)$ for $a>0$. For any proposal kernel $q_0$, let us define a corrected kernel by:
\begin{equation}
  \label{eq:corrected_kernel}
  q(x,v,\d v') = \psi\p{\frac{1}{2}\nabla  U(x)\cdot (v-v')} q_0(x,v;\d v'). 
\end{equation}

Our work is based on the following remark. 
\begin{Lem}\label{lem:invariance}
Assume that $q_0(x,v ; \d v')$ is reversible with
respect to $\gamma$, in the sense of condition~\eqref{eq:rev}. Let $q$ be the 
corrected non-normalized kernel defined by \eqref{eq:corrected_kernel}, 
where the function~$\psi$ satisfies~\eqref{eq:conditionOnPsi}. 
The corresponding operator $\calF$ given by~\eqref{eq:gen} satisfies the condition \eqref{eq:genF} if and only if the average condition~\eqref{eq:av} holds true; if this holds then 
the measure $\pi$ is invariant for the process. 
\end{Lem}

\begin{proof}
Let $a(x,v,v') = \frac{1}{2}\p{\nabla  U\cdot v-v'}$. 
 For any $\phi$, 
\begin{align*}
  \int \calF \phi(x,v) \d \pi(\d x \d v)
  &= 
  \int \phi(x,v')  \psi (a(x,v,v')) q_0(x,v;\d v') \pi(\d x \d v) \\
  &\quad - \int  \phi(x,v) \psi(a(x,v,v')) q_0(x,v;\d v') \pi(\d x \d v)\,.
\end{align*}
In the first integral, use the reversibility assumption and 
interchange the variables $v$ and $v'$ to get:
\begin{align*}
  \int \calF \phi(x,v) \d \pi(\d x \d v)
  &= 
  \int \phi(x,v)  \psi(a(x,v',v)) q_0(x,v;\d v') \pi(\d x \d v) \\
  &\quad - \int  \phi(x,v) \psi (a(x,v,v')) q_0(x,v;\d v') \pi(\d x \d v) \\
  &= \int \left[
     \int \p{\psi(a(x,v',v)) - \psi(a(x,v,v'))} q_0(x,v;\d v') 
  \right] 
  \phi(x,v) \pi(\d x \d v) \\
  &= -\int \left[
     \int a(x,v,v') q_0(x,v;\d v') 
  \right] 
  \phi(x,v) \pi(\d x \d v).
\end{align*}
As a consequence, the condition \eqref{eq:genF} is met if and only if the term between brackets is almost everywhere equal to $v\cdot\na U(x)$, 
  which is exactly the averaging condition~\eqref{eq:av}.  
\end{proof}

\begin{Rem}
In the case where the corrected kernel is constructed with
the function $\psi(s)=(s)_+$, at each jump, the scalar product of the velocity with
 $-\nabla U$ increases almost surely.  In that sense, there is ``minimal
 noise'' in the tangential part $\nabla U$. The condition can be relaxed by
 setting:
 \begin{equation}\label{eq:q_bis}
 q(x,v ; \d v') = \b{\psi \p{\frac{1}{2}\nabla  U(x) \cdot (v-v')  }  + g(x)} q_0(x,v ; \d v')
\end{equation}
for any non-negative function $g$ on $\R^d$. One checks easily that the averaging condition~\eqref{eq:av} is unchanged. The
process then performs jumps more often, but they are less constrained to be
aligned with $- \nabla U$. In fact, if $q_0$ is reversible for $\gamma$, then the kernel  $\tilde q(x,v ; \d v') :=  g(x)  q_0(x,v ; \d v')$ leaves invariant $\tilde \nu\otimes\gamma$ for all law $\tilde \nu$ on $\R^d$, and thus $\tilde q$ can be incorporated in the dissipative part $\mathcal D$ of the generator. For this reason, in the rest of the paper we only consider the case $g=0$.
\end{Rem}

\begin{Rem}\label{Rem:F*}
In the proof of Lemma \ref{lem:invariance} the integration with respect to the variable $x$ plays no role, so that in fact if $q_0$ satisfies the conditions (R) and (A) then for all $\phi\in\calC_c^\infty(\R^{2d})$,
\begin{equation}\label{eq:F*}
\int_{\R^{d}} \calF \phi  (x,v) \gamma(\d v) =  - \int_{\R^{d}}  \phi  (x,v) \p{v\cdot \nabla U(x)} \gamma(\d v) \qquad \text{$\d x$-a.e.}.
\end{equation}
\end{Rem}

\subsection{Particular known cases}
We now show how special forms of $q_0$ lead to various known sampling algorithms. 

\begin{The}[Zig Zag process]
  Let $\gamma$ be the uniform measure on the finite set $\{-1,1\}^d$. 
  For $v\in\{-1,1\}^d$, let $q_0(x,v;\d v') = \sum_{w: w\sim v} \delta_{w}(\d v')$, 
  where $w\sim v$ means that $v$ and $w$ are neighbours on the discrete
  cube, that is, they differ by one coordinate. 

  Then $q_0$ is reversible with respect to $\gamma$ and satisfies the
  average condition (A); the corresponding process is the 
  \emph{zig-zag process}. 
\end{The}

\begin{proof}
  The reversibility is clear. To check the average condition, 
  remark that if $w$ and $v$ differ only by the $i$\textsuperscript{th}
  coordinate, then $(v-w)/2 = v_i e_i$ where $e_i$ is the $i$\textsuperscript{th}
  basis vector. 
  Therefore
\[    \nabla  U(x)\cdot  \int_{v' \in \R^d} \frac{1}{2}(v - v') q_0(x,v ; \d v') 
 = 
 \sum_{i=1}^d  \nabla  U(x) \cdot v_i e_i   \\
=  \nabla  U(x)\cdot v. \]
With $\psi(s)=(s)_+$, the corrected kernel  is given by 
\[q(x,v,\d v') = \frac{1}{2}\p{ \nabla  U(x)\cdot (v-v')}_+ q_0(x,v;\d v') \\
  =  \sum_{i=1}^d \p{ \nabla  U(x)\cdot  v_i e_i}_+  \delta_{v - 2v_ie_i} \]
which is exactly the zig-zag jump kernel. 
\end{proof}

\begin{The}[Bouncy particle]
  Let $\gamma$ be the uniform measure on a sphere. For $v$ on the 
  sphere, let $q$ be the degenerate kernel 
  $q(x,v;\d v') = \delta_{R(x)v}(\d v')$ where $R(x)$ is the 
  symmetry with respect to the orthogonal of $\nabla U$, that is, 
  \[ R(x) v = v - 2 \one_{\{\nabla U \neq 0\}}   \frac{v\cdot \nabla U(x)}{|\nabla U(x) |^2} \nabla U(x).\]
  Then $q$ is reversible with respect to $\gamma$ and satisfies
  the average condition (A); the corresponding process 
  is the \emph{bouncy particle sampler}. 
\end{The}
\begin{proof}
  Once more, the reversibility is clear. 
  The interesting thing to notice here is that
  \[
    \frac{1}{2} \nabla U(x)\cdot (v-R(x)v) 
    =  \nabla U(x)\cdot v.
  \]
  Therefore
 \begin{align*}
   \nabla U(x) \cdot  \int_{v' \in \R^d} \frac{1}{2}(v - v') q_0(x,v ; \d v') 
=  \nabla U(x)\cdot v. 
\end{align*}
The corrected kernel with $\psi(s)=(s)_+$ is given by 
\(
  q(x,v,\d v') 
  = \p{ \nabla U(x)\cdot  v}_+ \delta_{R(x)v}(\d v'),
\)
and we recover the bouncy particle sampler. 
\end{proof}

\section{The Gaussian case}\label{Sec:Gaussian}

\subsection{The process}\label{Sec:Gaussian-def}
In this section we consider the particular case where the 
normalized proposal kernel $k_0(x,v;\d v')$  and the velocity distribution $\gamma$ are Gaussian. In fact, up to a change of variables, we assume without loss of generality that $\gamma$ is the standard Gaussian distribution with mean 0 and variance $\Id$.

Given the particular role of the direction $\nabla  U$, it is natural to
use the following orthogonal decomposition of the (tangent) space at $x$.
Denote $$\ta(x) = \na U(x)/|\na U(x)|$$ if $\na U(x) \neq 0$ and $\ta(x) = 0$ otherwise.
For any $w\in\mathbb R^d$, we write $w= w_T + w_O$ where $w_T = (w\cdot \ta(x)) \ta(x)$
is the projection of $w$ on $\Vect(\ta(x))$ and $w_O$ is orthogonal 
to $\ta(x)$. With this notation,  let $k_0(x,v;\cdot)$ be the distribution of 
the Gaussian random variable $V'$, defined by its 
decomposition $V'=V'_T+V'_O$: 
\begin{align}\label{eq:k_0}
  V'_T &=   \rho_T(x) v_T +  \sqrt{1-\rho_T^2(x)} G_T \\ \nonumber
  V'_O &=   \rho_O(x) v_O +  \sqrt{1 - \rho_O^2(x)}G_O
\end{align}
where $\rho_T(x)$, $\rho_O(x)$ are scalars in $[-1,1]$, 
and $G=G_T+G_O$ is a $d$-dimensional standard unit Gaussian. 
Recall that $k_0(x,v;\cdot)=\text{Law}(V')$ is (up to an 
intensity $\lambda_0$, see below) the proposal kernel for jumps in the velocity. One may therefore interpret the parameters $\rho_T$ and $\rho_O$ 
 as follows: 
\begin{itemize}
  \item the sign of $\rho_T$ encodes whether or not there is a "bounce", that is, a reflection of the component of the velocity that is tangent
  to the gradient of the log-density; 
  \item $\abs{\rho_T }$ encodes the strength of the memory
    for this tangential component: 
    if $\abs{\rho_T }=1$, the memory is perfect (the new tangential 
    component being either equal to the old one or to its opposite);
    on the contrary, $\rho_T(x)=0$ 
    means a full resampling without memory (which is called forward event-chain algorithm in~\cite{DurmusMichelSenechal});
  \item similarly $|\rho_O|$ and the sign of $\rho_O$ encodes respectively the balance between 
    full memory and full resampling,  and whether  or not  the orthogonal component of the velocity "bounces". 
\end{itemize}

It is then easy to remark:
\begin{Lem}\label{Lem:Gaussian-invariant}
 Let $q_0(x,v ; \d v') = \lambda_0(x,v) k_0(x,v ; \d v')$ where $k_0$ is
 defined above by~\eqref{eq:k_0}. For any $x,v\in\R^d$, the average condition~\eqref{eq:av} (equivalently \eqref{eq:av-2}) holds if 
 \[
 \lambda_0(x,v) = \frac{2}{1-\rho_T(x)};
 \]
and this latter condition is necessary when $v\cdot \nabla U(x) \neq 0$.
Moreover, if $\lambda_0$ does not depend on $v$,  then $q_0$ is reversible with respect to $\gamma$ (condition \eqref{eq:rev})).
\end{Lem}
\begin{proof}
To check the second form \eqref{eq:av-2} of the average condition, we compute for $x,v\in\R^d$
\begin{align*}
   \nabla  U(x)\cdot  \PAR{\int v' k_0(x,v;\d v') } 
  &= \nabla  U(x) \cdot \p{  \rho_T(x) v_T + \rho_O(x) v_0} \\
  &= \rho_T(x)  \nabla  U(x)\cdot v.
\end{align*}
Then, if $v\cdot \nabla U(x) \neq 0$, \eqref{eq:av-2} holds  iff   $1-2/\lambda_0(x,v) = \rho_T(x)$. 

Now, if $\lambda_0$ does not depend on $v$, the reversibility of $q_0$ is a consequence of the reversibility of $k_0$. Remark that the (density of the) kernel $k_0$ admits a decomposition $k_0(x,v;v')=k_0^T(x,v_T;v_T')k_0^O(x,v_O;v_O')$, and similarly $\gamma(v) = \gamma(v_T)\gamma(v_O)$, with
\begin{eqnarray*}
k_0^T(x,v_T;v_T')\gamma(v_T)& = &\frac{1}{2\pi \sqrt{1-\rho_T^2(x)}} \exp\left(-\frac{|v_T'- \rho(x) v_T|^2}{2( 1 - \rho_T^2(x))} -\frac{|v_T|^2}{2}\right)\\
& = & k_0^T(x,v_T';v_T)\gamma(v_T')
\end{eqnarray*}
and similarly for the orthogonal part, which concludes the proof of the reversibility.
\end{proof}

\begin{Rem}
  [Return of the bouncy sampler]
The degenerate, deterministic case $\rho_T = -1$, $\rho_O = 1$ gives
$\lambda_0 = 1$ and we get back the bouncy sampler. 
\end{Rem}
From now on we assume that there is no noise on the orthogonal part, that is, $\rho_O(x) =1$, and that $\lambda_0(x,v) = \lambda_0(x)  = 2/(1-\rho_T(x))$ for all $x,v\in\R^d$ and $\psi(s)=(s)_+$ for all $s\in\R$.  Introducing the notation
\begin{equation}
   \eps(x) := \frac{1-\rho_T(x)}{1+\rho_T(x)} \in [0,+\infty]
\end{equation}
we can express (dropping the $x$ dependence notation in the following for simplicity)
\[
\rho_T = \frac{1-\eps^2}{1+\eps^2}\,,\qquad \sqrt{1-\rho_T^2} = \frac{2\eps}{1+\eps^2}\,,\qquad  \lambda_0  = \frac{1+\eps^2}{\eps^2} .
\]
The consequences of the previous discussion are gathered in the following result.

\begin{Lem}[Velocity-jump sampler]\label{lem:velo}
Let $\eps = \eps(x) > 0$ denote a strictly positive function on $\R^d$. 
Let  $q_0$ be the proposal kernel given by
\[
\int \ph(v') q_0(x,v, d v') = \frac{1+\eps^2}{\eps^2} \E \p{\ph(V')},
\]
where the random variable $V'$ is constructed from the 
tangent vector $\ta = \ta(x) = \na U(x)/|\na U(x)|$ (if $\na U(x) \neq 0$, 
and $0$ otherwise), and a standard one-dimensional Gaussian~$G$ 
by the formula:
\begin{equation}\label{eq-V'}
 V' =   v-\frac{2\eps}{1+\eps^2} \p{ \eps v \cdot \ta +  G }   \ta. 
\end{equation}
Consider the PDMP generator $\calL=\calT+\calF$  where 
$\calT = v \cdot \nabla_x$ and $\calF$ is the velocity jump
operator given by correcting $q_0$: 
\[ 
\calF(\ph)(x,v) =   \int \p{\ph(v')-\ph(v)} q(x,v, d v')
= \frac{1}{2} \int \p{\ph(v')-\ph(v)}\p{\nabla_x U \cdot (v-v') }_+ q_0(x,v, d v'). 
\] 
 We have the following properties: 
 \begin{enumerate}
 \item the proposal kernel~$q_0$ satisfies the average condition~\eqref{eq:av} and is reversible with respect to 
 the unit Gaussian distribution in velocity variables.
 \item Consequently, the process with generator $\calL$ leaves 
 the target distribution~$\pi$ invariant. 
\end{enumerate}
\end{Lem}
In particular, $\eps=+\infty$ is the full bouncy particle, $\eps=1$ the full resampling,
$\eps<1$  a partial memory, and $\eps \to 0$ corresponds to small changes at an increasing jump rate. 

%
%
%
%
For theoretical and practical reasons, it is interesting to 
derive a more explicit  formula for  the corrected kernel.
\begin{The}[Corrected jump rate]\label{Theo:corrected-jump}
The corrected kernel associated with the velocity jumps process of Lemma~\ref{lem:velo} is given by
  \[ \int \phi(v') q(x,v, \d v') =  
  \frac{\abs{\nabla U}}{\eps}\esp{ \phi\p{ x,v -   \frac{2\eps}{1+\eps^2} \p{ \eps v \cdot \ta + G} \ta } \p{ \eps v \cdot \ta  + G}_+ }.\]
  As a consequence, the corrected jump rate $\lambda$ is given by
  \[
    \lambda(x,v) = \frac{\abs{\nabla U(x)}}{\eps(x)} \esp{%
      \PAR{ v \cdot \ta(x) \eps(x) +  G}_+} = \frac{ \abs{\nabla U(x)}}{\eps(x)} \Theta(\eps(x) v \cdot \ta(x))     
  \]
  where
  \[
    \Theta(u) := \esp{\p{u+G}_+}=  u \P(G > -u) + \frac{1}{\sqrt{2 \pi}} \e^{ - u^2/2}.
  \]
\end{The}

\begin{proof}
 By definition,
  \[ \int \phi(v') q(x,v, \d v') =  
  \frac12 \lambda_0(x)  \esp{ \phi\p{ x,V'} \p{ \nabla U \cdot (v-V')}_+ }.\]
  Replacing $V'$ by its expression given by \eqref{eq-V'} concludes.
\end{proof}

\subsection{Convergence toward the Hamiltonian dynamics}\label{Sec:Gaussian-limitHD}

From now on, we denote by $\calL_\eps$ the generator of the velocity jump process with kernel $q=q_\eps$ given by Theorem~\ref{Theo:corrected-jump} for some positive function $\varepsilon$ on $\R^d$, i.e.
\begin{equation}\label{eq:genLespi}
\calL_\eps \phi  (x,v) 
  = v\cdot \na_x \phi(x,v) + \int_{v' \in \R^d} \p{ \phi(x,v') - \phi(x,v)}  q_\eps(x,v; \d v')\,.
\end{equation}
It can then be formally expanded using Taylor's formula as 
\begin{equation}
 \calL_\eps(\phi) = v \cdot \nabla_x \phi + \sum_{n=1}^{+\infty} \frac{\abs{\nabla U} }{n !} D^n_v \phi \p{T, \ldots , T} \eps^{n-1} \p{ \frac{-2 }{1+\eps^2} }^n \esp{ \p{\eps v \cdot T +  G     }_+ ^{n+1} }\,.
\end{equation}
For $n\geqslant 1$,
\[\esp{\p{u+G}_+^{n+1}} = \esp{G_+^{n+1}}  + (n+1) u \esp{G_+^{n}}  + \underset{u\rightarrow 0}{\mathcal O}(u^2)\,,\]
with 
\[
\mathbb E(G_+) = 1/\sqrt{2\pi}\,,\qquad \mathbb E(G_+^2) = 1/2\,,\qquad \mathbb E(G_+^3) = \sqrt{2/\pi}\,.
\]
 As a consequence, as $\varepsilon$ vanishes, we formally get back the Hamiltonian dynamics
\[
\calL_\eps(\phi) = v \cdot \nabla_x \phi  -    \nabla U \cdot \nabla_y \phi + \mathcal O(\eps)\,,
\]
and at first order in $\eps$, we obtain a (degenerate) Langevin diffusion
\begin{align*}
\calL_\eps(\phi) &= v \cdot \nabla_x \phi  -    \nabla U \cdot \nabla_y \phi+ \varepsilon \frac{4 |\nabla U|}{\sqrt{2\pi}}\b{-(v\cdot T) T\cdot\nabla_v \phi +  D_v^2\phi(T, T) }+\mathcal O(\eps^2) \\
&= v \cdot \nabla_x \phi  -    \nabla U \cdot \nabla_y \phi+\eps  \frac{4 |\nabla U|}{\sqrt{2\pi}}\b{e^{\frac{|v|^2}{2 }}  T\cdot \nabla_v \p{e^{-\frac{|v|^2}{2 }} T\cdot \nabla_v \phi} } +\mathcal O(\eps^2)\,,
\end{align*}
which can be interpreted as the Langevin process  that is degenerate along the force direction;  is reversible (up to velocity reversal) with respect to the target distribution $\pi$, and has a typical relaxation time of order $1/(\abs{\nabla_x U} \eps  ) $.



\bigskip

We now give conditions under which the convergence of the velocity jump process towards an Hamiltonian dynamics can be proven rigorously. It is remarkable that the limit can be identified as soon as the martingale problem for the \emph{deterministic} Hamiltonian dynamics is well-posed. If $\nabla U$ is Lipschitz, this is a consequence of the standard Cauchy-Lipschitz theory; the minimal conditions on $\nabla U$ being still an open problem. We define first martingale problems in $\R^d$.

\begin{Def} A \cadlag  \ random process $\p{Z_t}_{t \geq 0}$ in $\R^{d}$ with initial distribution $\mu$ is solution to the martingale problem associated with $(\mu,L,C^\infty_c(\R^d))$, where $L$ is a Markov generator, if for any $\ph \in C^\infty_c(\R^d)$ the process
\[
t \mapsto \ph(Z_t) - \int_0^t L\ph (Z_s) d s
\]
is a martingale with respect to the natural filtration of $Z$. We say that uniqueness holds if all solutions have the same probability distribution on the usual Polish space of \cadlag\  trajectories.
\end{Def}
\begin{The}\label{th:conv}
Let $(\varepsilon_n)_{n\in\N}$ be a sequence of strictly positive measurable functions on $\R^d$ that vanishes uniformly on all compact sets as $n\rightarrow +\infty$. Denote $\mathcal L_{\varepsilon_n}$ the associated PDMP generator and denote
\[
\calL_0  \eqdef  v \cdot \nabla_x -  \nabla U \cdot \nabla_v \,,
\]
and consider $\mu \in \mathcal P(\R^{2d})$ an initial distribution. Assume that
\begin{itemize}
\item For each $n$, the velocity jump process associated to $\calL_{\varepsilon_n}$ is defined for all time (the sequence of jump times converges to $+\infty$).
\item $\nabla U$ is continuous and the martingale problem associated with $(\mu, \calL_0, C^\infty_c(\R^{2d}))$ is well-posed on $\R^{2d}$.
\end{itemize}

Then, as $n\rightarrow +\infty$, the velocity jump process associated to $\calL_{\varepsilon_n}$ converges in distribution in the space of \cadlag \  trajectories endowed with the Skorohod topology towards the unique martingale solution of the Hamiltonian dynamics $\calL_0$.
\end{The}
\begin{proof}
The proof follows from a general result, Theorem~\ref{th:conv_gen}, postponed to an Appendix section. Indeed, according to Theorem~\ref{th:conv_gen}, it is sufficient    to check that for any $\ph \in C^\infty_c(\R^{2d})$ and any compact $K \subset \R^{2d}$
\[
\lim_{n\rightarrow+\infty} \sup_{K} \abs{\calL_{\eps_n} \ph - \calL_0 \ph} = 0 .
\]
Using the definition of $\calL_{\eps_n}$ from Equation~\eqref{eq:genLespi}, and denoting by $u_n=u_n(x,v)=\eps_n(x) v\cdot \ta(x) + G$, 
(where $\ta(x)=\nabla U(x)/\abs{\nabla U(x)}$), 
the difference $(\calL_{\eps_n} \ph - \calL_0 \ph)(x,v)$ may be rewritten as: 
\[
 \esp{  \frac{\abs{\nabla U(x)} (u_n)_+}{\eps_n(x)}\p{ \phi\p{ v -   \frac{2\eps_n u_n(x)}{1+\eps_n(x)^2}  \ta(x) } - \phi(v)}  +   \na U(x)\cdot \na_v \phi(v) }. 
\]
Omitting the dependency in $x$ in the notations for legibility,
we may apply Taylor's theorem at the first order on the difference
$(\phi\p{ v -   \frac{2\eps_n u_n}{1+\eps_n^2}  \ta} - \phi(v))$ 
to get:
\begin{align*}
|(\calL_{\eps} \ph - \calL_0 \ph)(x,v)| 
& \leqslant  \left| 1- \frac{2 \esp{(u_n)_+^2}}{1+\eps_n^2}\right| |\na U\cdot \na_v \phi(v)  | + |\nabla U| \|\nabla^2 \phi\|_\infty \eps_n \esp{(u_n)_+^3}\,. 
\end{align*}
Since $\esp{(u_n)_+^2}$ converges to $1/2$ uniformly on all compact sets and $\esp{(u_n)_+^3}$ is uniformly bounded in $n$ on all compact sets, the right hand side 
vanishes uniformly on all compact sets of $\R^{2d}$ as $n\rightarrow +\infty$.
\end{proof}

\begin{Rem}
More generally, considering a limit generator $\calL_0+\calD_0$ for some dissipative $\calD_0$, the proof of Theorem~\ref{th:conv} is straightforwardly adapted to get the convergence of  the processes associated to generators $\calL_{\varepsilon_n} + \calD_{\varepsilon_n}$ with $\calD_{\varepsilon_n} \phi \rightarrow \calD_0\phi$ for all $\phi\in C^\infty_c(\R^{2d})$. For instance, that way we can  design velocity jump processes that converge toward the Langevin diffusion or the HMC process.
\end{Rem}

\subsection{Drift limit and factorization}\label{Sec-limit-factorize}

As discussed in Section \ref{sec:GeneralSetting}, if the forces are decomposed as $\nabla U(x) = \sum_{i=1}^N \xi_i(x)$ for some vector fields $\xi_i$, then we can consider the operators given by
\[\calF_i \phi  (x,v)  = \int_{v' \in \R^d} \p{ \phi(x,v') - \phi(x,v)}  q_i(x,v; \d v'),\]
with
  \[ \int \phi(v') q_i(x,v, \d v') =  
  \frac{\abs{\xi_i}}{\eps_i}\esp{ \phi\p{ v -   \frac{2\eps_i}{1+\eps_i^2} \p{ \eps_i v \cdot \ta_i +  G} \ta_i } \p{ \eps_i v \cdot \ta_i  +  G}_+ }\]
  where $G$ is a one-dimensional standard Gaussian variable, $x\mapsto \eps_i(x)$ is a positive function and $\ta_i(x) = \xi_i(x)/|\xi_i(x)|$ if $\xi_i(x)\neq 0$ and $\ta_i(x)=0$ otherwise.
  In other words, the process with generator $\mathcal T + \mathcal F_i$ is exactly the velocity jump process introduced in Section~\ref{Sec:Gaussian-def}, except that $\nabla U$ is replaced everywhere by $\xi_i$. 
  In particular, the previous results are straightforwardly extended: from Lemma~\ref{Lem:Gaussian-invariant}, the generators $\mathcal F_i$ satisfy \eqref{eq:genFi} (and more precisely \eqref{eq:F*} with $\nabla U$ replaced by $\xi_i$), so that $\mathcal L = \mathcal T + \sum_{i=1}^N \mathcal F_i$ satisfies $\int \mathcal L\phi d \pi = 0$ for all $\phi\in\mathcal C^\infty_c(\R^{2d})$. Similarly, from the computations of Section~\ref{Sec:Gaussian-limitHD}, 
  \[\mathcal F_i (\phi)  =   -    \xi_i\cdot \nabla_y \phi + \mathcal O(\eps_i)\,,\]
  and thus we still get the convergence toward the Hamiltonian dynamics, since
 \[ \calL(\phi) = v \cdot \nabla_x \phi  -     \nabla U \cdot \nabla_y \phi + \mathcal O(\underset{1\leqslant i\leqslant N}\max\eps_i)\,.\]
 Let us give two examples of such a factorization.
 
 \subsubsection{Gibbs velocity jump processes}\label{Sec-Gibbs}
 
 For $i\in\cco 1,d\ccf$, set $\xi_i(x) = \partial_{x_i} U(x) e_i$, where $e_i$ is the $i^{th}$ vector of the canonical basis and
 \[\calL = \calT + \sum_{i=1}^d \calF_i = \sum_{i=1}^d \p{ v_i \partial_{x_i} + \calF_i}\]
 where $\calF_i$ is defined as above, with some $\varepsilon_i$. The corresponding process can be seen as a (kinetic) Gibbs sampler: indeed, each generator $v_i \partial_{x_i} + \calF_i $ leaves invariant the conditional law $(x_i,v_i)\mapsto \pi(dxdv)$. 
 When $\eps_i(x) = +\infty$ for all $i$, we recover the zig-zag process, which may thus
 be seen as a Gibbs version of the bouncy   sampler (remark that, when $\varepsilon = +\infty$, the norm of the velocity is unchanged at jump times, so that  although $\pi = \nu\otimes\gamma$ is indeed invariant for $\calL$ with a a Gaussian distribution $\gamma$, it won't be ergodic).
 
 For a general choice of $\varepsilon_i$, this factorization ensures the following property: in the case where the target law is a tensor product of one-dimensional laws, i.e. if $U(x) = \sum_{i=1}^d U_i(x_i)$ for some one-dimensional  potentials $U_i$, then the coordinates of the corresponding kinetic Gibbs process are independent one-dimensional processes.
 
 Note that
 \begin{eqnarray*}
\calL_\eps(\phi) &=& v \cdot \nabla_x \phi  -     \nabla U \cdot \nabla_y \phi+  \sum_{i=1}^d \varepsilon_i \frac{4 |\partial_{x_i} U|}{\sqrt{2\pi}}\b{-v_i \partial_{v_i} \phi +   \partial_{v_i}^2\phi }+\mathcal O(\underset{1\leqslant i\leqslant N}\max\eps_i^2) \,.
\end{eqnarray*}
The fact that, in that case, the order one term is a non-degenerate Langevin diffusion is reminiscent of the fact the Zig-Zag process is irreducible in cases where the bouncy sampler is not, see \cite{BierkensRobertsZitt}.
 
 \subsubsection{Multi-time-stepping}\label{Sec-Multitimestep} 
 
 Suppose that $\na U=\xi_1+\xi_2$ where $\xi_1$ is large and numerically cheap to compute by comparison with $\xi_2$, smaller but numerically more intensive. To fix ideas, suppose that $\|\xi_i\|_\infty \leqslant L_i$ for $i=1,2$ with known constants $L_1\gg L_2$.   For $i=1,2$, take $\eps_i(x)=\eps_0$ for some $\eps_0>0$. Then,  in order to sample a trajectory of the process corresponding to the splitting $\na U = \xi_1+\xi_2$, as detailed in Section~\ref{Sec:Simulation}, $\xi_i$ will be computed at a rate $L_i/\eps_0$. Hence, the splitting reduces the number of computations of $\xi_2$.  This extends the strategy of \cite{Monmarche-MD} where $\varepsilon_1 = 0$ and $\varepsilon_2 = +\infty$ (bounce/drift process).

\subsection{Non-irreducibility}\label{Sec:non-irreductible}

The bouncy particle sampler and the Hamiltonian dynamics are well-known to be both non-irreducible in general. There are in fact non-irreducible counterexamples for all the processes with generator $\calL_\varepsilon = \calT + \calF_\eps$,  $\varepsilon >0$ in the case with no additional noise ($\calD = 0$). For instance, for a symmetric Gaussian target (or more generally any target with radial potential, i.e. that is invariant by isometries preserving the origin) in dimension larger than one, $\nabla U(X_t)$ being collinear to $X_t$, note that $X_t,V_t \in \mathrm{span}(X_0,V_0)$ for all $t\geqslant 0$. Moreover, assuming that $X_0$ and $V_0$ are not collinear, even within this two-dimensional plane, the process is not irreducible. Indeed, in the following, still for a symmetric Gaussian target, suppose that $d=2$ and $(x_0,v_0) \in \R^2\times\R^2$ with $\mathrm{span}(x_0,v_0)=\R^2$. Remark that $X_t\wedge V_t := X_t^1 V_t^2 - X_t^2 V_t^1$ is  unchanged by the free transport and by the jumps, hence is constant along time. In particular, starting from a deterministic condition $(x_0,v_0)$ the law of the process will never converge to the Gaussian target measure. More precisely, we expect the law of the process to converge to the law of a standard Gaussian variable $(X,V)$ on $\R^4$ conditioned to $X\wedge V = x_0\wedge v_0$ (since the standard Gaussian on $\R^4$ is invariant for the process, so is this conditional law). Even if we are only concerned with the law of $X$, this induces a bias (see the numerical section). \medskip

\section{Hypocoercivity}\label{Sec:hypoco}

The question of long-time convergence and ergodicity for velocity jump samplers have been addressed in various cases in \cite{BierkensRobertsZitt,Doucet2019,DurmusGuillinMonmarche2018_Bouncy} with a Meyn-Tweedie approach and in \cite{Andrieu} with the $L^2$ hypocoercivity method of Dolbeault-Mouhot-Schmeiser \cite{DMS}. Our approach will be similar to the latter. Since the process is not irreducible in general, a dissipative part is added for the velocities. In all this section, the target measure $\pi$ is given by \eqref{eq:gibbs} with $\gamma$ the standard (mean $0$, variance $\Id$) Gaussian distribution on $\R^d$ and  we consider a kinetic process with generator $\calL  = \calT  + \calF    + \calD$ as in Section \ref{sec:GeneralSetting} and $\calF $ is the operator defined in Lemma~\ref{lem:velo} for some non-negative function $\varepsilon$ on $\R^d$.

We would like to emphasize that we will only conduct a formal study, disregarding in particular the question of domains and extensions of the operators involved. The technical arguments to make the proofs valid would be exactly those of \cite{Andrieu}, and thus we omit them for the sake of clarity and in order to focus on the (formal) computations.

%
%
\begin{Ass}\label{AssumHypoco}
The dissipative part $\calD$ may be written as
$\calD = \eta(x)\calD_0$, where $\eta:\R^d\rightarrow \R_+$ is such that  
\[
\forall x\in\R^d\,,\qquad 0<\underline \eta<\eta(x) < \overline \eta \p{1+|\na U(x)|},
\]
for some $\overline \eta \geqslant \underline\eta>0$, and $\calD_0$ is a self-adjoint operator on $L^2(\gamma)$ such that $\calD_0(v)=-v$ and with a spectral gap of 1, in the sense that, for all nice $g\in L^2(\gamma)$,
\[
  \bracket{\calD_0 g,g}_{L^2(\gamma)}  \leqslant  -\| g -\int g \d \gamma\|_{L^2(\gamma)}^2\,.
\]
Moreover, $U\in\mathcal C^2(\R^d)$ and there exists $C_1\geqslant 0$ such that
\begin{equation}
\na^2 U(x) \succeq -C_1 I \tag{$C_1$}
\end{equation}
(in the sense of positive symmetric matrices) for all $x\in\R^d$ and
\begin{equation}\label{ConditionUHypoco}
\underset{|x|\rightarrow\infty}\liminf \p{\frac12 |\na U|^2 - \Delta U }  >   0\,.
\end{equation}
Finally, $\calT = v \cdot \na_x$ and $\calF$ belongs to the class of operators defined in Lemma~\ref{lem:velo}. 
\end{Ass}

\begin{Rem} The three classical dissipative operators $\calD_0$ given by~\eqref{eq:D1}, \eqref{eq:D2} and~\eqref{eq:D3} are all self-adjoint in $L^2(\gamma)$ with a spectral gap of 1 and with $\calD_0(v)=-v$.
\end{Rem}


The condition \eqref{ConditionUHypoco} classically implies that the measure $\nu$ satisfies a Poincar\'e inequality with some constant $c_P>0$: for all $f\in H^1(\nu)$,
\begin{equation}
\| f - \nu f\|^2_{L^2(\nu)} \ \leqslant \ \frac1{c_P} \| \na_x f\|_{L^2(\nu)}^2\,. \tag{$c_P$}
\end{equation}
It also implies that there exist $C_2>0$ such that 
\begin{equation}\forall x\in\R^d\,,\qquad \Delta U(x) \leqslant C_2 + |\na U(x)|^2/2\,.  \tag{$C_2$}
\end{equation}

In the following, $\|\cdot\|$ and $\bracket{\cdot}$ stands respectively for the norm and scalar product  in $L^2(\pi)$. 
 We denote by $m_2$ (respectively $m_4$) the second (respectively fourth) moment of $\gamma$: 
 \begin{align*}
 m_2 &= \int \abs{v}^2 d\gamma(v) = d, 
 &
 m_4 &= \int \abs{v}^4 d\gamma(v) = d(d+2). 
 \end{align*}
 Let $(P_t)_{t\geqslant 0}$ be the Markov semi-group with generator $\calL$.

\begin{The}[Exponential convergence in $\L^2$] 
\label{ThmHypoco}
 Under Assumption \ref{AssumHypoco}, for all $f\in L^2(\pi)$ and all $t\geqslant 0$,
\[ \|(P_t -\pi)f\| ^2 \ \leqslant \ \frac43 e^{-\kappa t }\|(I-\pi) f\| ^2 \,,\]
where $\kappa$ is given by:
\[\frac{1}{\kappa} \ = \ \frac{6}{\underline{\eta}} \p{1+ \frac1{ c_P^2}\p{1+\frac{C_1}{2 c_P }}\p{1+4C_2+16c_P^2} \p{\overline\eta /\sqrt{d} + 5\sqrt{1+2/d^2} + 4/d}^2}. \]

\end{The}
\begin{Rem} The main point here is that $\kappa$ does not depend on $\varepsilon$. Also note that, as a function of $\eta$, the convergence rate scales for large $\eta$ as $\underline{\eta}/\max(1, \overline{\eta}^2)$, which is well-known for the Langevin dynamics with a constant $\eta$ and suggests that the constant remains finite in the overdamped regime under proper rescaling (albeit with a sub-optimal constant of order $c_P^3$ instead of $c_P$).  For $c_P \ll 1$ and $d \gg 1$ we obtain
\[
\frac{1}{\kappa} \ \sim\  \frac{3}{c_P^3}\p{2c_P+C_1}\p{1+4 C_2 } \frac{1}{\underline{\eta}}\p{\overline{\eta}/ \sqrt{d} + 5}^2 .
\]
Alternatively, if $U$ is $\rho$-convex for some $\rho>0$ independent from the dimension (so that $C_1 = 0$ and $c_P=\rho$), choosing a constant $\eta = \sqrt d$, we get $\kappa = \mathcal O(C_2/\sqrt d)$. For instance, for a standard $d$-dimensional Gaussian target, $C_2=d$.
\end{Rem}
Denote $\mathcal M^*$ the dual of an operator $\mathcal M$ in $L^2(\pi)$, $\calS = (\calL+\calL^*)/2$ and $\calA=(\calL-\calL^*)/2$ the symmetric and skew symmetric parts of $\calL$ and
\[\Pi_v f(x,v) \ =\ \int f(x,v')\gamma(\d v')\,.\]
 The Dolbeault-Mouhot-Schmeiser method \cite{DMS} relies on the modified norm
\[\textbf{H}(f) =\frac12\| f\|^2 + \delta \bracket{\calB f,f}\,,\]
where $\calB$ is defined by
\[\calB   \ = \ -\p{m I + \p{\calA \Pi_v}^* \calA \Pi_v }^{-1} \p{\calA \Pi_v}^*\,,\]
for some scalar parameters $\delta,m>0$  to be chosen later on. From \cite[Proposition 26-(d)]{Andrieu} (applied to the operator $-\calA \Pi_v/\sqrt m$) , $\|\calB\|\leqslant 1/\sqrt m$ so that $\textbf{H}$ is equivalent to the $L^2(\pi)$ norm for $\delta<\sqrt m /2$. The aim is thus to prove that $\textbf{H}$ decays exponentially fast along the semi-group $(P_t)_{t\geqslant 0}$, which proves the hypocoercive decay in $L^2(\pi)$ (in the sense of \cite{Villani}, that is: exponential decay up to a constant factor $C>1$). Formally, the general result is the following:

\begin{The}\label{TheDMS}
Assume that 
\[\calS\Pi_v = 0\,\qquad \Pi_v \calA \Pi_v = 0\]
and that there exist $c_v,R(m)=R>0$ and $c_x(m)=c_x \in(0,1]$ such that, for all nice $f\in L^2(\pi)$ with $\pi f = 0$, it holds:
\begin{align}
&\text{(microscopic coercivity)}
&\bracket{\calS f,f} & \leqslant  - c_v  \|(I-\Pi_v) f\|^2 \label{EqMicroCoer}\\
&\text{(macroscopic coercivity)}
&\bracket{\calB\calA\Pi_v f,f} & \leqslant  -c_x \|\Pi_v f\|^2\label{EqMacroCoer}\\
&\text{(auxiliary bound)}
&\bracket{\calB\calL (1-\Pi_v)f,f}  & \leqslant  R\|\Pi_v f\| \|(I-\Pi_v)f\|\,. \label{EqAux}
\end{align}
Then, for all $f\in L^2(\pi)$ and all $t\geqslant 0$,
\[ \|(P_t -\pi)f\|^2 \ \leqslant \ \frac43 e^{-\kappa t }\|(I-\pi) f\|^2\,,\]
where 
\[ \kappa  \ = \ c_x \inf_{m > 0} \min\p{\frac{\sqrt m}{6},\frac{2c_v}{6+3R^ 2/c_x} }\,. \]
\end{The}
\begin{Rem} Typically, the macroscopic coercivity amounts to a spectral gap of the operator $ \p{\calA \Pi_v}^* \calA \Pi_v$ restricted to functions of space variables. In that case (which indeed occurs for our PDMP), one has
\[
c_x = \frac{c}{m+c},
\]
where $c$ is the spectral gap of $\p{\calA \Pi_v}^* \calA \Pi_v$. Then one can choose $m=c$ to get
\[
\kappa  \ = \ \min\p{\frac{\sqrt c}{12},\frac{c_v}{6+6 R^ 2 } } . 
\]
\end{Rem}
\begin{proof}
We only recall the main steps and refer to \cite{DMS,Andrieu} for details. Denoting $f_t = P_t f - \int f\d \pi$, from $\partial_t f_t =\calL f_t$ we get that
\begin{align*}
\partial_t \textbf{H}(f_t) & =  \bracket{f_t,\calL f_t} + \delta \bracket{\calB f_t,\calL f_t} + \delta\bracket{\calB\calL f_t,f_t}  \,.
\end{align*}
The microscopic coercivity condition \eqref{EqMicroCoer} intervenes in the first term
\[
\bracket{f,\calL f} \ = \ \bracket{f,\calS f} \ \leqslant \ -c_v \| (I-\Pi_v) f\|^2\,.
\]
Under the condition $\Pi_v \calA \Pi_v = 0$, the second term is bounded as
\[\bracket{\calB f,\calL f} \ \leqslant \ \|(I-\Pi_v) f\|^2\,,\]
see \cite[Lemma 5]{Andrieu}. From the macroscopic coercivity and auxiliary bounds conditions \eqref{EqMacroCoer} and \eqref{EqAux},
the third term gives
\[\bracket{\calB\calL f,f} \ =\ \bracket{\calB\calL \Pi_v f,f}   + \bracket{\calB\calL (1-\Pi_v)f,f}  \ \leqslant \ -c_x\|\Pi_v f\|^2 + R \|(I-\Pi_v)f\|\| \Pi_v f\|\,,\]
where we used that $\calS\Pi_v = 0$. Denoting $\alpha = \|\Pi_v f_t\|^2/\|f_t\|^2\in[0,1]$, we have thus obtained
\begin{align*}
\frac{\partial_t \textbf{H}(f_t)}{\|f_t\|^2} & \leqslant  (\delta-c_v)(1-\alpha) - c_x\delta \alpha + \delta R \sqrt{\alpha(1-\alpha)} \\
& \leqslant  \p{ \delta\p{1+\frac{R^2}{2c_x}}-c_v}(1-\alpha) - \frac 12 c_x\delta \alpha\,. 
\end{align*}
In particular, if $\delta \leqslant c_v /(2+R^2/ c_x)$, we get that 
\[\frac{\partial_t \textbf{H}(f_t)}{\|f_t\|^2} \ \leqslant \  - \frac12 c_v (1-\alpha) - \frac 12 c_x\delta \alpha \ \leqslant \   - \frac 12 c_x \delta \]
for all $\alpha\in[0,1]$, where we used that $c_x\leqslant 1$ and $\delta\leqslant c_v$. If moreover $\delta \leqslant \sqrt m /4$ we get that $\|f\|^2 \leqslant 4\textbf{H}(f) \leqslant 3 \|f\|^2$ and
\[\partial_t \textbf{H}(f_t) \ \leqslant \ - \frac 12 c_x \delta  \|f_t\|^2 \ \leqslant \ -\frac 23 c_x \delta \textbf{H}(f_t)\,.\]
We may then apply Gronwall's Lemma to conclude: for all $\delta\leqslant \min (\sqrt m /4,c_v /(2+R^2/ c_x))$,
\[ \|f_t\|^2 \ \leqslant \  4\textbf{H}(f_t) \ \leqslant \ 4 e^{-2c_x\delta t/3}\textbf{H}(f_0) \ \leqslant \ \frac43 e^{-2c_x\delta t/3}\\|f_0\|^2\,. \qedhere\]
\end{proof}

We now have to check that the conditions of Theorem \ref{TheDMS} are met under Assumption~\ref{AssumHypoco}. This is usually done by computing explicitly $\calS$ and $\calA$ for particular processes. In fact we will only need the following information, which is obtained from the   condition \eqref{eq:F*}, satisfied by all usual kinetic samplers:

\begin{Lem}
Under Assumption \ref{AssumHypoco}, 
\begin{eqnarray}\label{EqHypocoe}
\calS\Pi_v=0\,,\qquad \calA \Pi_v &=& \calT\Pi_v\,,\qquad \Pi_v \calA\Pi_v = 0\,.
\end{eqnarray}
\end{Lem}

\begin{proof}
Since $\calF$ and $\calD$ only act on the $v$ variable and $\Pi_v f$ only depends on $x$ for all $f$, $\calD\Pi_v =\calF\Pi_v  = 0$. Moreover, from condition \eqref{eq:F*}, for all  $f,g\in \mathcal C^\infty_c(\R^{2d})$,
\begin{align*}
\bracket{\calF^* \Pi_v f,g} & =  \int_{\R^{2d}} \calF g(x,v) \Pi_v f(x,v) \pi(\d x\d v) \\
& =  \int_{\R^d} \int_{R^d} f(x,w)\gamma(\d w) \int_{\R^d} \calF g(x,v) \gamma(\d v) \nu(\d x) \\
& =  - \int_{\R^{2d}} (v\cdot U(x)) g(x,v) \Pi_v f(x,v) \pi(\d x\d v)\,.
\end{align*}
In other words, $\calF^*\Pi_v f(x,v) =- v\cdot\na U(x) \Pi_v f(x,v)$. Besides,  $\calD^* = \calD$ by assumption and, integrating by parts, $\calT^* f(x,v) = -\calT f(x,v) + v\cdot\na U(x)  f(x,v)$. As a consequence,
\[2\calS\Pi_v  \ = \  (\calT +\calT^* +\calF+\calF^*+\calD+\calD^*)\Pi_v \ = \ 0\]
and
\[2\calA\Pi_v  \ = \  (\calT -\calT^* +\calF-\calF^*+\calD-\calD^*)\Pi_v \ = \ 2\calT\Pi_V\,.\]
Finally, for all $f\in\calC^\infty_c(\R^{2d})$,
\[
\Pi_v \calT\Pi_v f(x,v) \ = \   \int_{\R^d} w\gamma(\d w) \cdot \na_x \int_{\R^d} f(x,w)\gamma(\d w)  \ = \ 0\,.
\qedhere\]
\end{proof}



In particular, the operator $\calB$ being defined from the operator $\calA\Pi_v=\calT\Pi_v$, it is the same in our case and in \cite{Andrieu} (up to the choice of the parameter $m$, which is $m=m_2$ in \cite{Andrieu}). From \cite[Lemma 9]{Andrieu}, $\p{\calA \Pi_v}^* \calA \Pi_v f = m_2 \na_x^* \na_x \Pi_v   f$ and thus
\begin{eqnarray}\label{EqDefu}
\calB^* f &=& -\calT u\,,\qquad\text{where}\qquad u\ = \  \p{m I + m_2\na_x^*\na_x}^{-1} \Pi_v f\,.
\end{eqnarray}
Remark that $u$ is a function of $x$ alone.

\begin{Lem}\label{LemHypoco1}
Under Assumption \ref{AssumHypoco}, the  microscopic and macroscopic coercivity conditions \eqref{EqMicroCoer} and \eqref{EqMacroCoer} respectively hold with $c_v=\underline\eta$ and $c_x = m_2 c_P/(m+m_2 c_P)$. 
\end{Lem}

\begin{proof}
To get the microscopic coercivity estimate, we remark that $\calT+\calF$ is the generator of a Markov semigroup that fixes $\pi$, so that
\[0 \ \geqslant\  \int f(\calT+\calF)f \pi \ = \ \frac12 \int f(\calT+\calF+\calT^*+\calF^* f)\pi\,, \]
and thus
\begin{eqnarray*}
\bracket{\calS f,f} \ \leqslant \ \bracket{\calD f,f} & = & \int \eta(x) \int f(x,v) \calD_0 f(x,v) \gamma(\d v)\nu(\d x) \\
& \leqslant & -\int \eta(x) \int \p{f(x,v) -\Pi_v f(x,v)}^2 \gamma(\d v)\nu(\d x)  \leqslant \ - \underline\eta \| f-\Pi_v f\|^2\,.
\end{eqnarray*}
For the macroscopic condition, remark that 
\[\calB \calA \Pi_v f = -\Phi\p{\p{\calA \Pi_v}^* \calA \Pi_v   }f \]
with $\Phi(z)  = z/(m+z)$, which is a non-decreasing function from $\R_+$ to $[0,1]$. Moreover, $\p{\calA \Pi_v}^* \calA \Pi_v $ is self-adjoint and for all $f \in L^2(\pi)$ such that $\pi f =0$ (so that $\nu\Pi_v f=0$),
\[\bracket{\p{\calA \Pi_v}^* \calA \Pi_v f ,f} \ = \ m_2 \bracket{\na_x^* \na_x \Pi_v   f, \Pi_v f} \ = \ m_2 \| \na_x \Pi_v  f \|^2  \ \geqslant \ m_2 c_P \|\Pi_v f\|^2\,.\]
From the spectral mapping theorem \cite[Theorem 2.5.1, Corollary 2.5.4]{davies_1995}, $\Phi\p{\p{\calA \Pi_v}^* \calA \Pi_v   }$ is self-adjoint with a spectral gap bounded by $\Phi(m_2c_P)$, which concludes.
\end{proof}

The previous results have been established using only the general condition~\eqref{eq:F*}. 
By contrast, the proof of the auxiliary bound~\eqref{EqAux} is based on the particular form
of~$\calL$.
\begin{Lem}\label{LemHypoco2}
Under Assumption \ref{AssumHypoco}, 
\[   \bracket{\calB\calL (1-\Pi_v)f,f}  \ \leqslant \ R \|\Pi_v f\| \|f-\Pi_v f\|\,,\]
where $R$ is given by 
\[R^2 = \frac1{m_2}\p{\frac1{m_2}+\frac{C_1}{2m  }}\frac{1+4C_2+16c_P^2}{c_P^2} \p{\overline\eta \sqrt{m_2} + 5\sqrt{m_4} + 4}^2  \,. \]
\end{Lem}

\begin{proof}
First, we bound
\[ \bracket{\calB\calL (1-\Pi_v)f,f} \ = \ \bracket{(1-\Pi_v)f,\calL^*\calB^*f} \ \leqslant \ \|(I-\Pi_v)f\| \|\calL^*\calB^* f\|\,. \]
Let $u$ be defined by \eqref{EqDefu}. Using the process definition in Lemma~\ref{lem:velo}, we first remark that since i) $\calT + \calF$ conserves the target distribution $\pi$ and ii) $q_0$ is reversible, one has:
\[
\p{\calT + \calF}^* \ph (x,v) = - v \cdot \na_x \ph + \int \p{ \ph(x,v') - \ph(x,v) } \p{\nabla U \cdot (v'-v) }_+ q_0(x,v,d v').
\]
Using that $\calD_0(v)=-v$ and that $(\na U \cdot T)(\na_x u\cdot T) = \na U\cdot\na_x u$,  
\begin{eqnarray}
\calL^* \calB^* f & = & - \calT^2 u - \na_x u \cdot \calD(v) + \frac12  \na_x u \cdot \int \p{ \na U\cdot (v'-v)}_+ (v'-v)   q_0(x,v;\d v') \notag \\
& = & - v\cdot \na_x^2 u \, v +   \eta v\na_x u - \frac{2}{1+\varepsilon^2} (\na U\cdot\na_x u)\int \p{\varepsilon v\cdot T + w}_-^2 e^{-w^2/2}\frac{\d w}{\sqrt{2\pi}}\notag \\
& =: & - v\cdot \na_x^2 u \, v +  \eta v\na_x u -  (\na U\cdot\na_x u)H(v,x)\notag
\end{eqnarray}
(Recall that $\varepsilon$, hence $H$, can depend on $x$). We bound
\begin{eqnarray*}
H(v)  & \leqslant & \frac{2}{1+\varepsilon^2} \int \p{\varepsilon v\cdot T +   w}_-^2 e^{-w^2/2} \frac{\d w}{\sqrt{2\pi}}\ \leqslant \ 4 |v|^2 + 4 \, . 
\end{eqnarray*}
As a consequence,
\[|\calL^* \calB^* f| \ \leqslant \ |v|^2 |\na_x^2 u| + \overline\eta |v| |\na_x u| \sqrt{1+|\na U|^2}  + \p{4|v|^2 + 4}|\na U||\na_x u|\,, \]
and
\[\|\calL^* \calB^* f\| \ \leqslant \ \sqrt{m_4} \|\na_x^2 u\| +   \p{\overline\eta \sqrt{m_2} + 4\sqrt{m_4} + 4}\| \sqrt{1+|\na U|^2}\na_x u\|\,. \]

 Finally, the following elliptic  regularity estimates are  proven in \cite[Corollary 35 and Proposition 33]{Andrieu}:
\begin{eqnarray*}
\| \na_x^2 u\|^2  & \leqslant & \p{\frac1{m_2^2}+\frac{C_1}{2m m_2 }} \|\Pi_v f\|^2 \\
 \|\sqrt{1+|\na U|^2} \na_x u\|^2 & \leqslant & \p{\frac1{m_2^2}+\frac{C_1}{2m m_2  }}\frac{1+4C_2+16c_P^2}{c_P^2} \|\Pi_v f\|^2 \,,
\end{eqnarray*}
 which concludes using $\frac{1+4C_2+16c_P^2}{c_P^2} \geq 1$.
\end{proof}

We may now conclude the proof. 

\begin{proof}[Proof of Theorem \ref{ThmHypoco}]
  By \eqref{EqHypocoe} and Lemmas \ref{LemHypoco1} and \ref{LemHypoco2}, Theorem \ref{TheDMS} applies with any choice of $m>0$.  
   We take $m=m_2 c_P$, so that $c_x = 1/2$ in Lemma~\ref{LemHypoco1},   $R^2$ given in Lemma~\ref{LemHypoco2} is
  \begin{align*}
  R^2& =  \frac1{m_2^2 c_P^3}\p{c_P+\frac{C_1}{2  }}\p{1+4C_2+16c_P^2} \p{\overline\eta \sqrt{m_2} + 5\sqrt{m_4} + 4}^2
  \end{align*}
  while one has from Theorem \ref{TheDMS} 
\[  \kappa \ = \  \min\p{\frac{\sqrt{m_2c_P}}{12},\frac{\underline{\eta}}{6(1+R^2)}}\,.
\]
Recall $m_2=d$ and $m_4=d(d+2)$. Let us show that the minimum is always given by the second term. Using that $C_1,C_2 \geqslant 0$, we simply bound
\[
\frac{6(1+R^2)}{\underline{\eta}} \ \geqslant \ \frac{6R^2}{\overline{\eta}} \ \geqslant \   \frac{6 \p{1+16c_P^2}\p{\overline\eta / \sqrt{d} + 5}^2}{c_P^2\overline\eta}\,.
\]
Optimizing with respect to $\overline\eta$   we remark that $  (\overline\eta/\sqrt{d} + 5)^2 / \overline{\eta} \geqslant 20/\sqrt{d}   $. Moreover, we always have $(1+16c_P^2)/c_P^2 \geqslant 1/\sqrt{c_P}$, and thus $6(1+R^2)/\underline{\eta} \geqslant 120/\sqrt{dc_P}$. As a conclusion, $\kappa = \underline{\eta}/(6+6R^2)$.
\end{proof}



\section{Simulation of velocity-jump processes}\label{Sec:Simulation} 
\subsection{General strategy}
The practical implementation of our 
velocity jumps processes rely on two assumptions:
\begin{enumerate}[label=\roman*)]
  \item  the gradient $\nabla U(x)$ of the log-density can be computed numerically,
  \item some prior estimates on $\nabla U$ are given, typically its uniform norm or global
    Lipshitz constant.
\end{enumerate}
For the sake of simplicity we only consider the case $\psi(s)=(s)_+$, although the extension to other cases is straightforward.

In order to simulate exactly a velocity jump-process we need some \emph{a priori}
information on the jump rate evolution. 
\begin{Def} Let $\lambda(x,v) = \int_{v'} q(x,v; \d v')$ be the total jump rate of a velocity
 jump process. A function $\bar{\lambda} : \R^d \times \R^d \times
  \R^+ \to \R^+$ is called a \emph{prior rate upper bound} if
\[
 \lambda(x+tv ,v)  \leq  \bar{\lambda}(x,v,t)    \qquad \forall x,v \in \R^d, \, t \in \R^+.
\]
\end{Def}

The simulation of the process is based on increasing the number of
jumps at the price of adding uneffective (also called ghost) jumps.  
The jump times and velocities at those jump times, which determine the whole trajectory,
 are defined by induction.  The simulation of the jumps then follows the algorithmic rules:
\begin{enumerate}[label=(\roman*)]
  \item At time $t$, compute $t+S$ the next jump time so that
  \[
   \int_0^S \bar{\lambda}(X_t,V_t, s) \d s = E,
  \]
where $E$ is independent unit exponentially distributed.
The expression of the prior rate bound $\bar{\lambda}$ shall be sufficiently
simple to compute $S$ cheaply and exactly (up to round-off). 

\item Note that $X_{(t+S)-}= X_t+S V_t$ and $V_{(t+S)-} = V_t$. With probability 
\[
\frac{\lambda(X_t+S V_t,V_t)}{\bar{\lambda}(X_t,V_t,S)}
\]
sample a new  velocity $V_{t+S}$ according to the probability kernel 
$k(X_t+S V_t, V_t;\d v')$; else do not change velocity.
 \end{enumerate}
 
 In the rest of this section, we present how a suitable prior rate upper bound can be established and how to sample according to $k$ in the case of the Gaussian velocity jump samplers introduced in Theorem~\ref{Theo:corrected-jump}.

\subsection{Bounds on the corrected rate}
Consider the jump rate $\lambda$ defined in Theorem~\ref{Theo:corrected-jump}. The bound $ \E[ (a + b G)_+] \leq (a)_+ + b /\sqrt{2 \pi}$ yields
  \[    \lambda(x,v) \leqslant  (v\cdot \na U(x))_+ + \frac{  \abs{\nabla U(x)}}{\sqrt{2\pi} \eps(x)}\,.\]  
  Natural choices for $\varepsilon(x)$ are $\eps(x) = \eps_0 |\nabla U(x)|$ (as this gives a uniform bound on the second part of the jump rate), $\varepsilon(x) = \eps_0$ and $\varepsilon(x) = \eps_0/(1+|\nabla U(x)|)$ (for which, according to the discussion in Section~\ref{Sec:Gaussian-limitHD}, the degenerate Langevin term that appears as the first order error with respect to the Hamiltonian dynamics as $\varepsilon\rightarrow 0$ is then uniformly bounded in $x$). In any of those cases, a prior rate upper bound can be obtained from bounds on $v\cdot \nabla U(x+tv)$ and $|\nabla U(x+t)|$. Such bounds are easily obtained if $\nabla U$ is uniformly bounded by some known constant $L$, or if the the Hessian $H$ of $U$ is globally bounded in the Euclidean matrix norm, i.e. $M := \sup_{x \in \R^d} \|H(x)\|_2 < \infty$, in which case
\[  v\cdot  \nabla U(x + vt)  \leqslant   v\cdot  \nabla U(x)+ M |v|^2 t\,,\qquad |\nabla U(x+t)| \leqslant |\nabla U(x)| + M |v| t\,.\]
Each of the three choices of $\varepsilon$ above yields a bound of the form
\[\lambda(x+tv,v) \leqslant \bar\lambda(x,v,t):=  M |v|^2 (t-t_0(x,v))_+ + a(x,v) + b(x,v) t^k \]
for some  $k\in\{1,2\}$ and $a,b,t_0\geqslant 0$. Remark that, from the properties of the exponential law, then
\[S := \inf\left\{s>0, \int_0^s \bar{\lambda}(X_t,V_t, s) \d s > E\right\}\]
has the same law as $S_1\wedge S_2\wedge S_3$ where, denoting $\bar \lambda_1=M |v|^2 (t-t_0)_+$, $\bar \lambda_2 = a$ and $\bar\lambda_3 = b t^{k}$,
\[S_i := \inf\left\{s>0, \int_0^s \bar{\lambda_i}(X_t,V_t, s) \d s > E_i\right\}\,,\]
for $i=1,2,3$, where $E_1,E_2,E_3$ are independent with unit exponential distribution. Here,
\[S_1 = t_0 + \sqrt{\frac{2E_1}{M|v|^2}}\,,\qquad S_2 = \frac{E_2}{a}\,,\qquad S_3 = \left(\frac{(k+1)E_3}{b}\right) ^{\frac1{k+1}}\,.\]

%
%
%

\subsection{Sampling according to the corrected kernel}
\subsubsection{General strategy}
Consider the process defined in Theorem~\ref{Theo:corrected-jump}.
Then, omitting in the notation the dependency
of $\eps$ and $\ta$ on $x$, 
the velocity after jump is $v - 2\eps( \eps v \cdot \ta +
 \tilde G)\ta/(1+\eps^2)$ where $\tilde G$ is a one-dimensional
random variable with density
\[
  f_m(y) = \frac 1 {\Theta(m) \sqrt{2 \pi}} (m +  y)_+ \exp\left( -y^2/2 \right),
\]
where $m = \eps v\cdot \ta $. We sample $\tilde{G}$ using rejection sampling,
with various proposal distributions, depending on the value of the parameter~$m$.
In order to fix notations, we briefly recall the procedure. 
We look for  a function $g_m$ satisfying the two requirements:
\begin{enumerate}
\item $g_m$ is a probability density from which we know how to sample;
\item there exists $C_m>0$ such that for all $x$, $f_m(x) \leq C_m g_m(x)$
  and the ratio $f_m(x)/(C_m g_m(x))$ is computable. 
\end{enumerate}
The rejection sampling then consists in drawing $Y$ according to $g_m$,
and accepting it with probability $f_m(Y)/(C_m g_m(Y))$, and repeating
until a proposal is accepted. It is well-known that this leads to a sample
distributed according to $f_m$, and that the number of proposals needed is geometrically distributed
with mean~$C_m$. 

\subsubsection{Proposal distributions}
We now list various choices for the proposal distribution with the corresponding
computations; these choices are compared in terms of the expected number of trials
and the CPU time in our implementation below. 

\paragraph{Gamma proposal}
For $m<0$, one can choose  a $\Gamma(2,-m)$ proposal, shifted by $(-m)$:
\[
  g_m(y) =  (y+m)_+ (-m)^2 \exp\left( - (-m)(y+m)\right),
\]
which is the distribution of $(-m)+ (E_1 + E_2)/(-m)$, where $E_1$ and
$E_2$ are standard exponential random variables.  This choice yields
\begin{align*}
  \frac{f_m(y)}{g_m(y)}
  &= \frac{1}{\sqrt{2\pi} m^2\Theta(m)} \exp\left( - y^2/2 - m(y+m) \right) \\
  &= \frac{1}{\sqrt{2\pi} m ^2\Theta(m)} \exp\left( - (y+m)^2/2 - m^2/2\right), 
\end{align*}
which is less than $C_m = \exp(-m^2/2) / (\sqrt{2\pi} m^2\Theta(m))$.
A proposed value $y$ is accepted with probability $\exp(-(y+m)^2/2)$,
and the expected number of trials $C_m \to 1$ for $m\to -\infty$.

\paragraph{Exponential proposal}
Still for  $m<0$, we can use an exponentially distributed proposal, shifted by $(-m)$:
 \[
   g_m(y) = \lambda \exp\left( -\lambda (y+m)\right) \one_{y>-m}.
\]
The choice $\lambda = -m$ leads to simple bounds:
\begin{align*}
  \frac{f_m(y)}{g_m(y)}
  &= \frac{1}{(-m)\Theta(m)\sqrt{2\pi}} (m+y)_+ \exp\left(-y^2/2 - m(y+m)\right)
\end{align*}
is maximized for $y=(-m)+1$, so $f_m(y) \leq C_m g(y)$ where
\begin{align*}
  C_m =  \frac{1}{(-m)\Theta(m)\sqrt{2\pi}}  \exp\left(-1/2 - m^2/2\right).
\end{align*}
The acceptance  probability in $y$ is
\begin{align*}
  \frac{f(y)}{C_m g(y)}
  &=(-m)  (m+y) \exp\left( -y^2/2 - my - m^2 + 1/2 + m^2/2\right)\\
  &=(-m) (m+y) \exp(1/2) \exp\left( - (y+m)^2/2\right).
\end{align*}
The constant  $C_m \sim \exp(-1/2) (-m)$ is unbounded for $m\to -\infty$.
However it behaves better than the Gamma proposal for small values of $\abs{m}$. 

\paragraph{Shifted  Rayleigh proposal}
Consider once more the case $m<0$. In the density $f_m$,  $(m+y)_+$ is then
bounded above by $y\one_{y>-m}$, leading to the bound 
\[
  f_m(y) \leq  \frac 1 {\Theta(m) \sqrt{2 \pi}} \one_{y>-m} y \exp\left( -y^2/2 \right) = C_m g_m(y),
\]
where
\begin{align*}
  C_m     &= \frac{\exp(-m ^2/2)}{\sqrt{2\pi} \Theta(m)}, &
   g_m(y) &= \one_{y>-m} y\exp\PAR{m ^2/2 - y^2/2}
\end{align*}                                                   
It is easily checked that $g_m$ is the distribution of
$\sqrt{m^2 + 2E}$ for $E$ an exponentially distributed random variable. 

From the expansion $\mathbb P(G\geqslant x) \simeq \exp(-x^2/2)(1/x-1/x^3)/\sqrt{2\pi}$ as $x\rightarrow \infty$, we get
the asymptotic behaviour
\[
  C_m  =  m^2 + \underset{m\rightarrow - \infty}o(m^2)
\]
implying that this choice is bad when $\abs{m}$ is large.
On the contrary, $C_m$ converges to the optimal value $1$
when $m$ goes to $0_-$.

\paragraph{Mixture between Rayleigh and Gaussian distribution}
We now turn to the case $m>0$ and bound $(m+y)_+$
from above by $m+y\one_{y>0}$. 
\[
  f_m(y) \leq  \frac 1 {\Theta(m) \sqrt{2 \pi}} (m +  y\one_{y>0}) \exp\left( -y^2/2 \right) := C_m g_m(y) 
\]
where $C_m=(m+1/\sqrt{2\pi})/\Theta(m)$ and  $g_m$ is a probability density.
One easily checks that $g_m$ is the density of the mixture 
\[
  \tilde Y = G\ \one_{U \leqslant m/(m+1/\sqrt{2\pi})} + \sqrt{2E}\  \one_{U > m/(m+1/\sqrt{2\pi})}
\]
where $G$, $E$ and $U$ are independent and respectively distributed
according to the standard Gaussian law, the standard exponential
distribution and the uniform law over $[0,1]$; it is therefore easy to sample. 
The proposal is accepted with probability $ (m+Y)_+/ (m +  Y\one_{Y\geq 0}) $.

The bound 
\[
  \Theta(m) \geq  \mathbb E((m+G)\one_{G\geq 0})
  = \frac m2 + \frac 1 {\sqrt {2 \pi}}
\]
shows that $C_m$ is always less than $2$ and  converges to $1$ when $m$ vanishes. For $m\to\infty$, $\Theta(m)\sim m$ and $C_m \to 1$.  

\paragraph{Gaussian proposal}
If $m>0$, the mode of $f_m$ is $\alpha = (\sqrt{m^2+4} -m)/2$. Let $g_m$ be the density
of the Gaussian random variable $\mathcal{N}(\alpha,1)$. Then
\begin{align*}
  \frac{f(x)}{g(x)}
  &= \frac{1}{\Theta(m)} (m+y)_+ \exp\left( -y^2/2 + (y-\alpha)^2/2\right) \\
  &= \frac{1}{\Theta(m)} (m+y)_+ \exp\left( -\alpha y + \alpha^2/2\right).
\end{align*}
This is maximized for $y+m = 1/\alpha$, leading to the bound
\begin{align*}
  \frac{f_m(x)}{g_m(x)} \leq C_m = \frac{1}{\Theta(m) \alpha} \exp\left( - \alpha^2/2\right).
  \end{align*}
  The algorithm then consists in sampling from $g_m$ and accepting with probability
  \[ \frac{f(x)}{C_m g_m(x)} = \alpha (m+y)_+\exp\left( -\alpha y + \alpha^2 \right). \]

  If $m$ goes to infinity, $\alpha \sim 1/m$, so $C_m\sim m/\Theta(m) \to 1$.
  If $m$ goes to zero, $\alpha$ goes to $1$, and $C_m$ to $\exp(-1/2)\frac{1}{\Theta(0)}
  = \sqrt{2\pi} \exp(-1/2)\approx 1.52$.

\subsubsection{Choice of the proposal}
We compare in Figure~\ref{fig:proposals} the various choices for the proposal distributions, 
both theoretically and empirically. The best method depending on $m$ will of course depend 
on implementation details; the important point is that by choosing an appropriate proposal we 
are able to keep the expected number of samples before acceptance $C_m$ bounded. 
For our implementation we are led to 
the following choices. 

\begin{figure}
\centering
\includegraphics[width=0.45\linewidth]{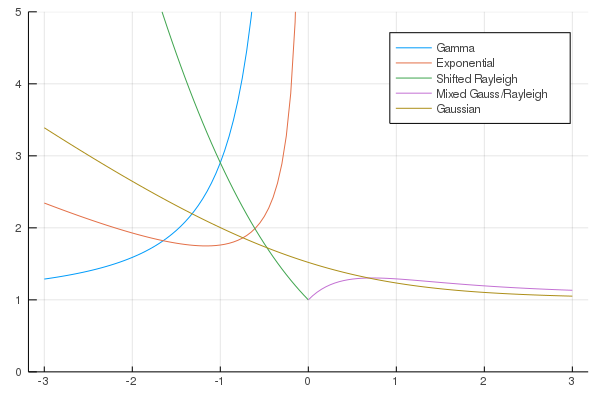}
\includegraphics[width=0.45\linewidth]{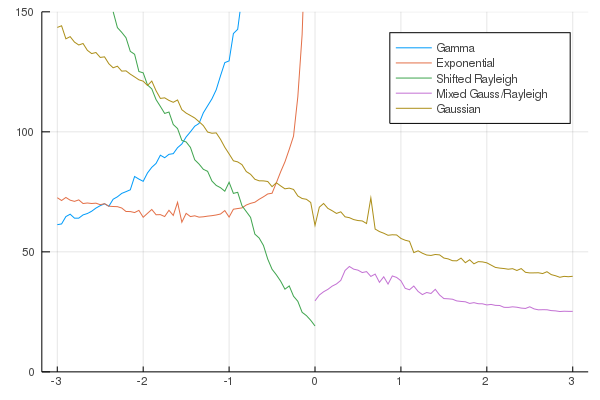}

\scriptsize
On the left, we plot the value of $C_m$, the expected number of samples before acceptance, as a function of $m$, for the five proposal distributions discussed above. On the right we plot the empirical time (in nanoseconds) used by our implementation of the various methods. Note that the Gaussian proposal is in practice, for our implementation, a little slower than its competitors. From both point of views, the minimum of the curves stays uniformly bounded. 
\caption{Comparison of proposal distributions}
\label{fig:proposals}
\end{figure}        
    
     \begin{center}   
        \begin{tabular}{cc}
\toprule
          $m$ & Best proposal \\
          \midrule
          $m\lesssim 2.5$ & Gamma  \\
          $-2.5\lesssim m \lesssim -1$ & Exponential \\
          $-1 \lesssim m \lesssim 0$ & Rayleigh \\
          $0\lesssim m $ & Mixed Rayleigh/Gaussian
          \end{tabular}
\end{center}

\section{Numerical experiments}\label{Sec:numerique}
We provide in this section a numerical illustration for the very simple case of the two dimensional unit Gaussian distribution. We choose the precision parameter to be constant $\varepsilon(x)=\varepsilon$, and the simulated process is the velocity-jump process described in Lemma~\ref{lem:velo}, without any additional noise on velocity.

\paragraph{Motivation}
Although this example may seem {\it a priori} na\"ive, it  is motivated by the practical problem of sampling according to distributions with ``multiscale'' densities in Euclidean space. Indeed, near a local minimum, the potential (log-density) is approximately quadratic, which justifies the choice of the potential. Moreover, the few fastest time scales of the process -- corresponding to stiffest directions of the local minimum -- typically cannot be identified, and may be considered decoupled from: i) other degrees of freedom, and ii) additional noise on velocity which is usually restricted to the slowest time-scale. Those fastest degrees of freedom are the ones we arguably emulate here.

\paragraph{Simulation parameters}
Simulations are carried out with the following parameters:
\begin{itemize}
\item An initial condition $(x_0,v_0) \in \R^4$.
\item A number of force evaluations $n \geq 1$.
\item A quadratic potential of the form:
\[
V(x) = x_1^2/2+ \lambda x_2^2/2
\]
with asymmetry parameter $\lambda  \geq 1$. $\lambda = 1$ corresponds to the potential with (vectorial) isometry symmetry.
\item A constant dynamical precision parameter $\varepsilon(x)=\varepsilon$ (see Lemma~\ref{lem:velo}).
\end{itemize}

\paragraph{Irreducibility issues} Without additional noise (which provides not only ergodicity but also exponentially fast mixing, see Section~\ref{Sec:hypoco}), the simulated velocity-jump process may not be irreducible with respect to the normal distribution (see Section~\ref{Sec:non-irreductible}). In the present section, we will observe the following two cases.
\begin{itemize}
\item The invariant distribution is the unit normal distribution, hence it is invariant by origin preserving isometries. In that case, the process is not irreductible, and it is easy to check that $t \mapsto X_t \wedge V_t$ is constant through time ($x \wedge v =x_1 v_2 - x_2 v_1$ in an orthonormal basis so that $x \wedge v =0$ if and only if $x$ and $v$ are collinear). The process seems to be irreducible with respect to the unit normal $(X,V)$ conditioned by $X \wedge V = x_0 \wedge v_0$ and $X,V \in \Vect(x_0,v_0)$ where $(x_0,v_0)$ are the initial conditions of the process.
\item The invariant distribution is an asymmetric normal distribution, and the process seems to be irreducible in dimension $2$ in that case.
\end{itemize}
Rigorous analysis of irreducibility issues without additional noise is left for future work.

 
\paragraph{Results --- short trajectories}

In Fig.\ref{fig:traj} and~\ref{fig:traj2} we plot short/medium time trajectories for $\lambda =1$ (the unit, symmetric quadratic potential $\abs{x}^2$) and initial condition is $x_0=(1,0)$, $v_0=(1,1)$. Total physical time is (roughly) constant, so that the number of force evaluations increases with the precision parameter $\eps$. We observe that when $\eps \to 0$, trajectories indeed converge to the expected Hamiltonian dynamics of a two dimensional harmonic oscillator (integrated with a Verlet scheme here). 

\begin{figure} 
\centering
\begin{minipage}{0.3\linewidth}
\includegraphics[scale=0.3]{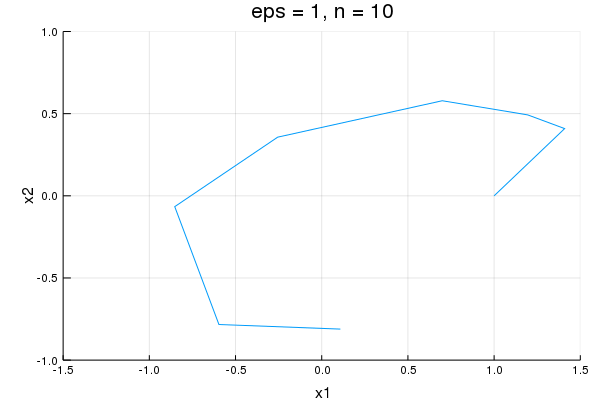}
\end{minipage}
\begin{minipage}{0.3\linewidth}
\includegraphics[scale=0.3]{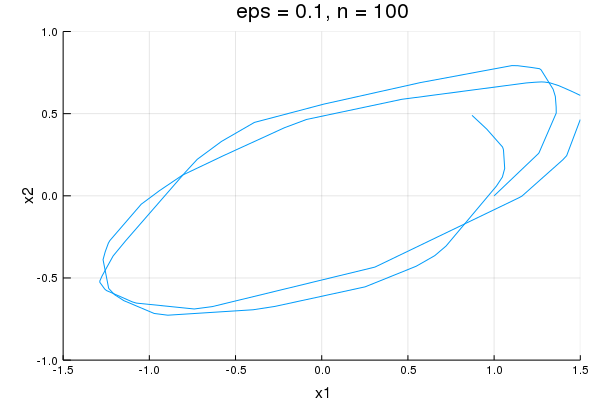}
\end{minipage}
\begin{minipage}{0.3\linewidth}
\includegraphics[scale=0.3]{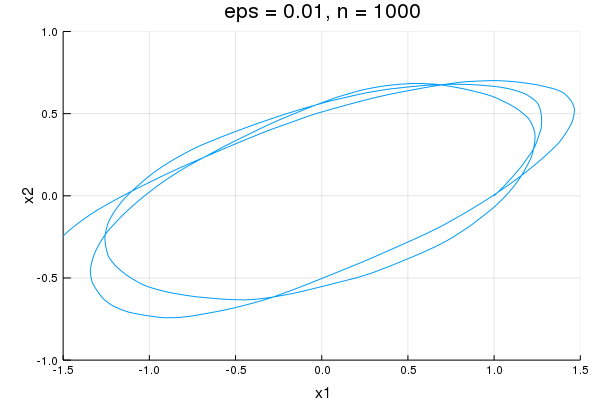}
\end{minipage}

\begin{minipage}{0.3\linewidth}
\includegraphics[scale=0.3]{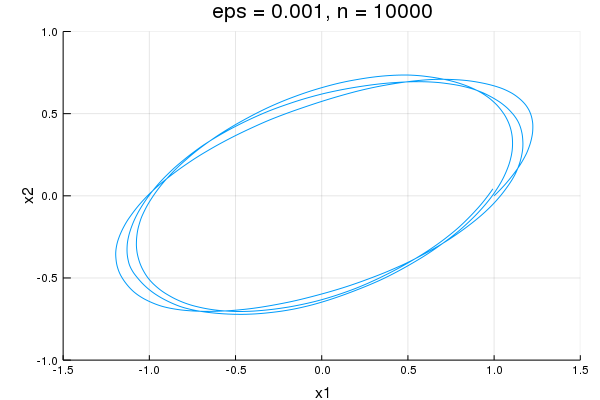}
\end{minipage}
\begin{minipage}{0.3\linewidth}
\includegraphics[scale=0.3]{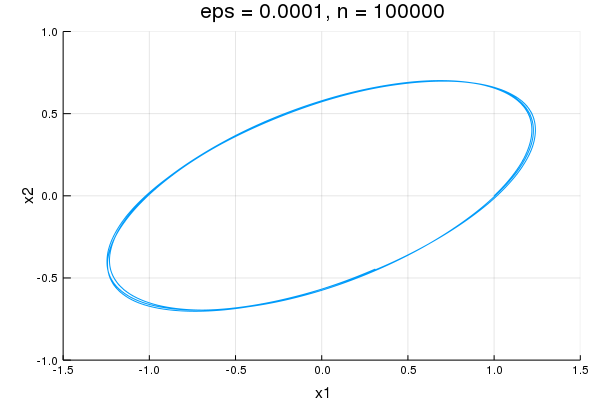}
\end{minipage}
\begin{minipage}{0.3\linewidth}
\includegraphics[scale=0.3]{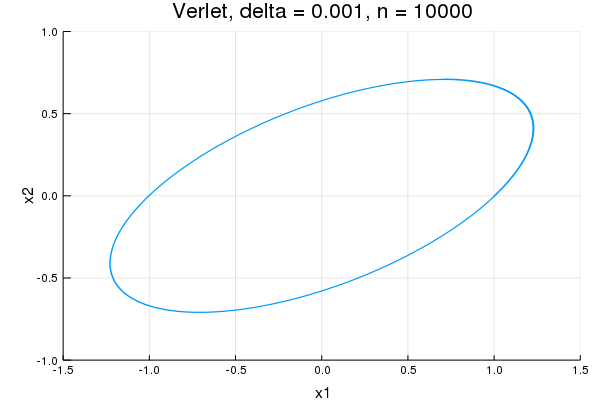}
\end{minipage}
\caption{Examples of trajectories with the same (approximate) time length. The velocity-jump process is compared to the Hamiltonian limit computed with a Verlet scheme. Various $\eps$ are compared.\label{fig:traj}}
\end{figure}

\begin{figure}
\centering
\begin{minipage}{0.4\linewidth}
\includegraphics[scale=0.3]{images/jump00001-100000}
\end{minipage}
\begin{minipage}{0.4\linewidth}
\includegraphics[scale=0.3]{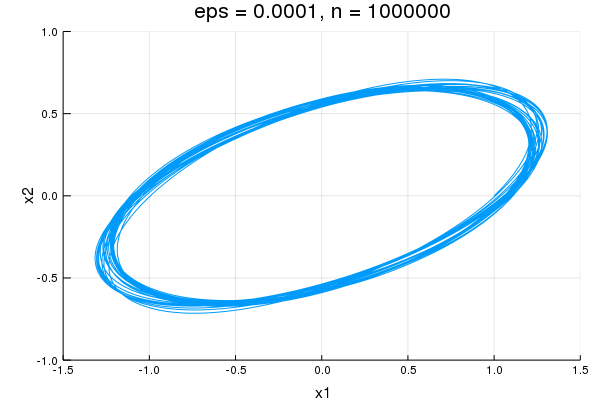}
\end{minipage}
\caption{Same as \ref{fig:traj} but for a longer trajectory \label{fig:traj2}}
\end{figure}

\paragraph{Results --- long non-ergodic trajectories}

In Fig.\ref{fig:traj3} we plot long time trajectories for $\lambda =1$ (the unit, symmetric quadratic potential $\abs{x}^2$), initial condition $x_0 = (0,0.5)$,
$v_0 = (0.5,0)$, and total number of force evaluations $n=10^5$. The expected non-ergodicity is observed.

\begin{figure} 
\centering
\begin{minipage}{0.3\linewidth}
\includegraphics[scale=0.3]{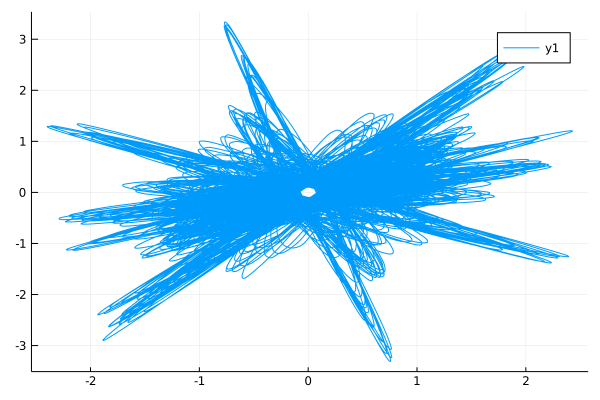}
\end{minipage}
\begin{minipage}{0.3\linewidth}
\includegraphics[scale=0.3]{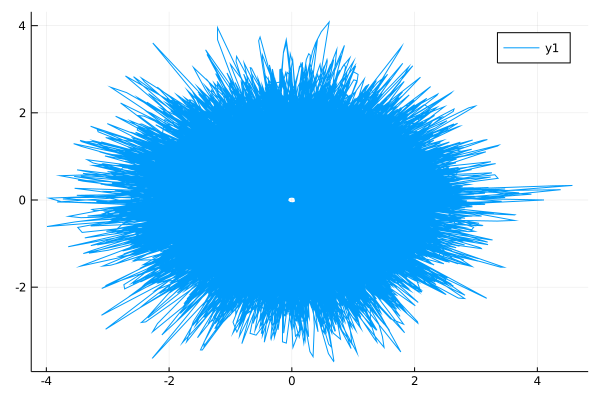}
\end{minipage}
\begin{minipage}{0.3\linewidth}
\includegraphics[scale=0.3]{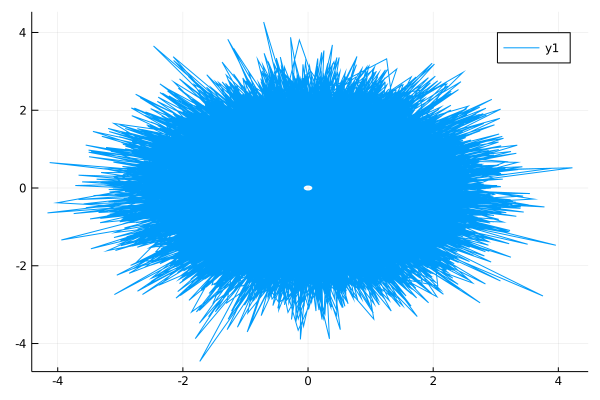}
\end{minipage}

\caption{Examples of long trajectories for the (non-irreducible) unit Gaussian for, from left to right, $\eps \in\set{.01,1,100}$.\label{fig:traj3}}
\end{figure}

\paragraph{Results -- mixing}

In Fig.\ref{fig:mix}, we fix the initial condition $x_0 = (0,0.5)$, $v_0 = (0.5,0)$, and the number of force evaluations $n=10^5$. We consider the position observable given by the time average of the square distance to the origin
$$
\frac1T \int_0^T \abs{X_t}^2 dt .
$$
For this observable, we compare the mixing efficiency for various $\eps  \in \{10^{-2},10^{-1},1,10,10^2\}$ and $\lambda \in \{1, 1.05, 5\}$ using various independent  samples obtained by simulating the velocity-jump process. Let us recall that $\eps =0$ corresponds to the Hamiltonian dynamics, while $\eps = + \infty$ is exactly the bouncy sampler. The figure consists of three (left, right, bottom) groups of box plots of those samples (each corresponding to a value of $\lambda$), the horizontal axis being $\ln \eps$.  

\begin{figure}

\centering

\begin{minipage}{0.4\linewidth}
\includegraphics[scale=0.3]{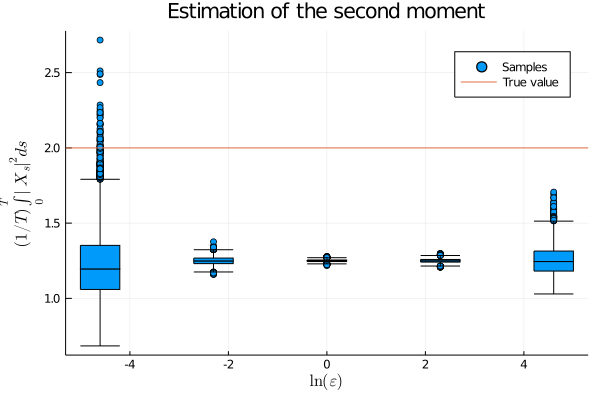}   \\ \centering $\lambda=1$
\end{minipage}

\begin{minipage}{0.4\linewidth}
\includegraphics[scale=0.3]{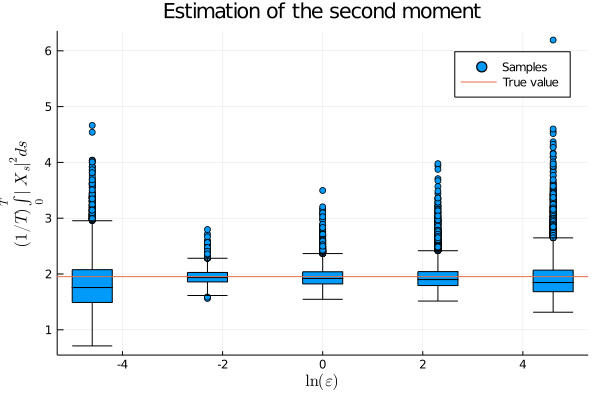} \\ \centering $\lambda=1.05$
\end{minipage}
\begin{minipage}{0.4\linewidth}
\includegraphics[scale=0.3]{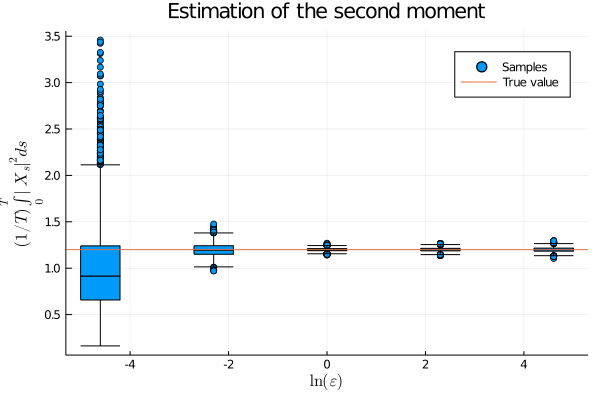}  \\ \centering $\lambda=5$
\end{minipage}

\caption{Box plots of samples obtained with fixed number of force evaluation $n=10^5$. Comparison between:  $\eps \in \{10^{-2},10^{-1},1,10,10^2\}$ on
the horizontal axis, as well as symmetric versus asymmetric potentials --- eigenvalues ratio $1$ (up chart), $1.05$ (left chart) and $5$ (right chart). Observe: i) the bias due to lack of ergodicity in the symmetric case, iii) a decrease of variance in the $\lambda =5$ very asymmetric case, and  iii) an efficiency which seems optimal for various non-extremal values of $\eps$. \label{fig:mix}}

\end{figure}

Several remark and results:
\begin{itemize}
\item As expected, for $\lambda =1$ (and in this case only), the process is not irreducible and the sample is biased. A quick calculation shows that if $(X,V) \in \R^4$ is unit Gaussian then 
$$
\E(|X|^2 | X \wedge V = c) = 1 + c K_1(c) /K_0(c)
$$
where $K$ denotes the modified Bessel special function of the second kind. With our choice of initial conditions, $c =1/4$ and the above quantity is roughly $1.224 $ which is consistent with the observed bias.
\item For $\lambda =1.05$, the proximity to $\lambda=1$ where the breakdown of irreducibility (conservation law) occurs, seems to result in a larger variance than the other cases.
\item For very small value of $\eps$, the process: i) is simulated on comparatively shorter timescale due to the required precision, ii) is closed to an Hamiltonian dynamics which possesses additional conserved quantities (in particular energy). This translates into a poor mixing and thus a larger variance.

\item We observe that the optimal sample quality is obtained for various intermediate values of $\eps$ ( $\eps=1$ or even lower, that is full tangential resampling or closer to the Hamiltonian limit). These intermediate cases seems to consistently outperform the bouncy sampler ($\eps = + \infty$).

\end{itemize}

\paragraph{Conclusion} The process  exhibits irreducibility issues in the presence of radial symmetries that are similar to the ones of the bouncy sampler. A moderate addition of velocity noise is thus recommended in general. The optimal sampling efficiency seems to be obtained for intermediate values of $\eps$, around $1$ or a bit lower (closer to the Hamiltonian limit than the full resampling, but not too much). However, this particular optimal value seems to vary with the target model and requires further and more exhaustive analysis.



\section{Supplementary material}\label{Sec:Supplement}
In this section, we establish a general result on the convergence of a family of Markov processes,
Theorem~\ref{th:conv_gen}, which is used in the proof of Theorem~\ref{th:conv}.

Consider a family $(L_\eps)_{\eps > 0}$ of Markov
generators on $\R^d$, and Markov processes $(X^\eps_t)$ associated to
these generators by a martingale problem. There is a large literature (a reference monograph we will abundantly use here is~\cite{EK}) linking convergence properties of $(L_\epsilon)$ with a
convergence at the level of stochastic processes. Our purpose is to
provide a simple generic setting in which checking the convergence of
$L_\epsilon$ to a limiting generator $L$ \emph{locally} is enough to imply weak convergence at the process level. The classical 'weak' (convergence in distribution) approach of~\cite{EK} relies on characterization of Markov processes by their generator through martingale problems (see below). Applying the convergence of generators at the level of the martingale problem typically enables to obtain tightness of the process distributions, extract a limit from them, and identify it.

In order to state the result, let us briefly recall that 
if $E$ is a Polish state space, the set of \cadlag{} trajectories
indexed by $\R_+$ may be equipped with the Skorokhod topology,
forming a Polish space denoted by $\D_E$
(Section~$5$ and~$6$,
Chapter~$3$ of~\cite{EK}).  We also recall that a sequence of
\cadlag{} trajectories $x^n$ 
converges to $x$ in $\D_E$ if, on any finite time interval, it converges
uniformly up to a uniformly small time change.

\begin{The}\label{th:conv_gen}
  Let $\p{\p{X^\eps_t}_{t \geq 0}}_{\eps>0}$ denotes a family of \cadlag{}
  processes in $\R^d$ with initial distribution $\mu$. Assume the
  following:
\begin{itemize}
\item For each $\eps$, $\p{X^\eps_t}_{t \geq 0}$ solves the martingale
  problem associated with $\p{\mu,L_\eps,C^\infty_c(\R^d)}$.
\item For all $\ph \in C^\infty_c(\R^d)$, $L_\eps \ph$
  converges to $L \ph$ uniformly on compacts.
\item $L \ph$ is continuous and the martingale problem associated
  with $\p{\mu,L,C^\infty_c(\R^d)}$ is well-posed in $\R^d$
  (in particular the solution exists for all time) for any initial
  probability distribution $\mu$.
\end{itemize}
Then $X^\eps$ converges in distribution in the Skorohod space
towards the unique solution of the limiting martingale problem.

\end{The}
\begin{Rem}
  The case of $\R^d$ could be easily generalized to any
locally compact Polish space.
\end{Rem}
\begin{proof}
  \newcommand{\proofstep}[1]{\par\medskip\emph{#1}\hspace{1em}}
  The proof uses heavily the  classical technical
  apparatus developed in~\cite{EK}. Let us first give an outline
  of the strategy before going into details.

  The key point in order to use the ``local'' convergence
  of $L_\epsilon$ to $L$ is to stop the processes when they
  leave large compact sets of $\R^d$, say balls defined by
\[
B_r \eqdef \set{x, \, \abs{x} \leq r},
\]
and  to remark that the family of stopped processes is tight
with respect to the Skorohod topology.
Using the convergence of $L_\eps$ to $L$,  any limit of extracted
$\eps$-sequences is then shown to coincide, when stopped, with
the unique solution of a stopped martingale problem
associated with $L$.
In the last step, stopping the processes outside an appropriate
ball, the global convergence is established.  

Let us now give details on these three steps.

\proofstep{Tightness for stopped processes.}
If $F \subset E$ is
closed, and $x \in \D_E$ we consider the hitting time
\[
  \tau(F) = \tau(F,x) \eqdef
  \inf \set{t \geq 0, \, \abs{x_{t}} \in F \text{\,or\,} x_{t^-} \in F} \in [0,+\infty],
\]
which is a stopping time for the canonical natural filtration of the
Borel sets of $\D_E$.
We fix an $r>0$, let $F=\{x: \abs{x}\geq r\}$ and consider 
$X^{\eps,F}$ the stopped process
\[ X^{\eps,F}(t) = X^\eps(t \wedge  \tau(F)).\]
The goal of this first step is to prove that  $(X^{\eps,F})_{\eps >0}$,
whose trajectories stay in the bounded set $\set{x:\abs{x}\leq r}$, is tight. 
The proof follows a very classical pattern; we sketch
it using~\cite{EK} as reference for the sake of completeness. Details can be found in~\cite{rousset2019weak}, Section~$3.2$.

Using \cite[Theorem~$9.1$, Chapter~$3$, p.$142$]{EK},
tightness is equivalent to the tightness in $\D_\R$ of
$\p{\ph(X^{\eps,F})}_{\eps >0}$ for each $\ph \in
C^\infty_c(\R^d)$. Fix $\ph \in C^\infty_c(\R^d)$. Classically: i)
expand squares of the form
$\p{\ph(X^{\eps,F}_{t+h}) - \ph(X^{\eps,F}_t)}^2$;
ii) consider the
two (stopped) martingales associated with $\ph(X^{\eps,F}_t)$ and
$\ph^2(X^{\eps,F}_t)$; and
iii) use the uniform boundedness on compacts
of $L_\eps \ph$ and $L_\eps \ph^2$ (which follows from the convergence
assumption).
Standard tightness criteria like \cite[Theorem~$8.6$,
Chapter~$4$, p.$137$]{EK} enables to conclude.

\proofstep{Identification of the limit through a stopped martingale problem.}
Let $X^n=X^{\epsilon_n, F}$ be an arbitrary convergent subsequence of $X^{\epsilon,F}$, and call its limit $Y$. Call $X$ a solution of the limit martingale problem; recall that we assume well-posedness so $X$ is unique in distribution. 

We claim that the stopped process $Y^\tau = Y(t\wedge \tau(F))$ solves
a \emph{stopped martingale problem}
(\cite[Section 6, Chapter 4]{EK}): for any $\phi$,
\begin{equation}
  \label{eq:limitMP}
  \phi(Y(t\wedge \tau(F))) - \phi(Y(0))
  - \int_0^{t\wedge \tau(F)} L\phi(Y(s)) ds
\end{equation}
is a martingale with respect to the natural filtration of $Y$. 
By \cite[Theorem~$6.1$ p.~$217$ Ch.$4$]{EK}, there is a unique solution of this
stopped martingale problem, namely the distribution of $X$ stopped at $F$, so that:
\begin{equation}
  \label{eq:identificationOfY}
  Y^\tau  \stackrel{(d)}{=} X^\tau \eqdef X(\cdot \wedge\tau(X,F)).
\end{equation}

Let us now justify the claim.
By Lemma~\ref{lem:choosingNeighborhoods} below, there exists a sequence $\delta_n$
such that, denoting by $F_n$ the $\delta_n$-neighborhood $F(\delta_n)$
of $F$, the sequence  $(X^n,\tau(F_n,X^n))$
converges in distribution towards $\p{Y,\tau(F,Y)}$.
In particular, as can be seen by a Skorohod almost sure representation of the latter convergence, we get the convergence in distribution of $\tilde{X}^n(\cdot) = X^n(\cdot\wedge \tau(F_n,X^n))$:
\[
  (\tilde{X}^n, \tau(F_n, \tilde{X}_n))
  \xrightarrow[(d)]{}
  (Y^\tau, \tau(F,Y)).
\]
Now $\tilde{X}^n$ solves the martingale problem associated to $L_{\epsilon_n}$, stopped at time $\tau(F_n)$: let us briefly see how to send $n$ to infinity and justify the claim.

By \cite[Lemma~$7.7$, Chapter~$3$, p.~$131$]{EK}
there exists a  dense subset of times $\calC \subset \R$ where
the limit $Y^\tau$ is continuous with probability one.
Let $t_1$,... $t_{K+1}$ be arbitrary times in $\calC$ and $\phi$, $\phi_1$,
... $\phi_K$ be bounded test functions. By definition of the stopped
martingale problem solved by $\tilde{X}^n$ and the characterization
of martingales given in~\cite[p.$174$]{EK}, 
\begin{equation}
  \label{eq:altMPn}
  \esp{%
    \p{ \phi(\tilde{X}^n(t_{K+1})) - \phi(\tilde{X}^n(t_K))
      - \int_{t_K\wedge \tau(F_n,\tilde{X}^n)}^{t_{K+1}\wedge\tau(F_n,\tilde{X}^n)} L_{\eps_n} \phi(\tilde{X}^n(s))ds}
    \prod_{k=1}^K \phi_k(\tilde{X}^n(t_k))
  } = 0.
\end{equation}
The left-hand side may be written as $\esp{ \Phi(\tilde{X}^n, \tau(F_n, \tilde{X}^n))}$
for some function $\Phi$. 
Remarking by dominated
convergence that since $L \ph$ is continuous, integrals of the form $x
\mapsto \int_0^t L \ph (x_s) d s $ are continuous with respect to the
Skorokhod topology, and since the $t_k$ are in $\calC$, $\Phi$ is almost
surely continuous at the limit $(Y^\tau,\tau(F,Y))$. This justifies taking the
limit in~\eqref{eq:altMPn}, which yields
\[  
  \esp{%
    \p{ \phi(Y^\tau(t_{K+1})) - \phi(Y^\tau(t_K))
      - \int_{t_K\wedge \tau(F,Y)}^{t_{K+1}\wedge\tau(F,Y)} L\phi(Y^\tau(s))ds
      }
    \prod_{k=1}^K \phi_k(Y^\tau(t_k))
  } = 0.
\]
Using again the previously mentioned characterization of martingales, this
entails that $Y^\tau$ indeed satisfies the martingale problem with generator~$L$,
stopped at time $\tau = \tau(F,Y)$.

\proofstep{Convergence of the original processes.}
We fix a bounded continuous observable $\Psi(x)=\Psi(x_s, 0 \leq s
\leq T)$ on $\D_E$ mesurable with respect to paths restricted to a
given finite time interval $[0,T]$.
Since the limit martingale problem is assumed to be well-posed in $\R^d$, the solution $X$ exists for all time, and thus
for each $\eta>0$ there exists a $r = r(T,\eta)$ such that
denoting $F=\{x, \abs{x}\geq r\}$, 
\[
\P( \tau(F,X) \leq 2T ) \leq \eta . 
\]

Our goal is to prove that
\(  \abs{\E\b{\Psi(X^\eps)} - \E\b{\Psi(X)}} \to 0\), or in other words that
\[
  D \eqdef
\limsup_{\epsilon\to 0}
  \abs{\E\b{\Psi(X^\eps)} - \E\b{\Psi(X)}}
\]
is zero. Let us extract a sequence $X^n = X^{\epsilon_n}$ such that
\(   D = \lim_n
  \abs{\E\b{\Psi(X^n)} - \E\b{\Psi(X)}}
\);
up to extracting a further subsequence we may assume by tightness that
$X^{n,F}$ converges in distribution.
By \eqref{eq:limitMP} and \eqref{eq:identificationOfY} from
the previous step, there exists a sequence $(\delta_n)$ such that,
for $F_n = F(\delta_n)$, 
\begin{equation}
  \label{eq:cvgProcessusArretes}
  (\tilde{X}^n, \tau(F_n,\tilde{X}^n)) \xrightarrow[]{(d)} (X^\tau, \tau(F,X)),
\end{equation}
where $\tilde{X}^n=X^{n,F_n}$ is the process stopped when it reaches $F_n$.
Since $X^n(t)=\tilde{X}^n(t)$ when $t<\tau(F_n,X_n)$,  
\begin{align*}
  \abs{\E\b{\Psi(X^n)} - \E\b{\Psi(X)}}
  &\leq \abs{
    \esp{\Psi(\tilde{X}^n) \one_{\tau(F_n,\tilde{X}^n)  > T}}
     - \esp{\Psi(X^\tau)\one_{\tau(F,X) > T}}
    } \\ 
  & \quad + \norm{\Psi}_\infty \P(\tau(F_n,\tilde{X}^n) \leq T   ) \\
  &\quad  + \norm{\Psi}_\infty \P(\tau(F,X)  \leq T)
\end{align*}
By~\eqref{eq:cvgProcessusArretes}, the first term vanishes in the limit so 
$D \leq 2\norm{\Psi}_\infty \eta$. Since $\eta$ is arbitrary, $D$ must be zero,
concluding the proof of convergence. 
\end{proof}

\newcommand{\prb}[1]{\mathbb{P}\left[#1\right]}

The above proof uses a technical result to 
handle the fact that, for a given closed set $F$,
the map $x\mapsto \tau(F,x)$ is only lower semicontinuous with respect to the Skorokhod topology.
To understand what may go wrong, consider $X^n$ the deterministic motion in $\mathbb{R}$ that goes upwards or downwards at speed one and is reflected on the boundary
of $[-1+1/n, 1-1/n]$: for $F=\mathbb{R}\setminus]-1,1[$, the hitting time of $F$ is infinite
for $X^n$ but finite for the limiting process~$X$; in particular, the stopped
process $X^n(t\wedge \tau(\{-1,1\},X^n)) = X^n(t)$  does \emph{not} converge to $X\p{t\wedge\tau\p{\set{-1,1},X}}$. 

We first prove a deterministic result showing that we may almost recover continuity by
considering $\delta$-neighborhoods of $F$, $F(\delta) = \{x: d(x,F)\leq \delta\}$.  
\begin{Lem}
  \label{lem:sko}
  Suppose that $x^n \to x$ in $\mathbb{D}_E$ and let $F$ be a closed set. Then
  $\delta\mapsto \tau(F(\delta),x)$ is decreasing and
  \begin{align}
    \label{eq:lemSko1}
    \limsup_n \tau(F(\delta),x^n) &\leq \tau(F,x) \leq \liminf_n \tau(F,x^n), \\
    \label{eq:lemSko2}
    \tau(F(\delta),x) &\xrightarrow[\delta\to 0]{} \tau(F,x). 
  \end{align}
  Consequently for any  sequence $\delta_n\to 0, $
  \begin{equation}
    \label{eq:lemSko3}
    \liminf_n  \tau(F(\delta_n), x^n ) \geq  \tau(F,x).
    \end{equation}
\end{Lem}
Note that we can  only expect to get a statement on the liminf if the
sequence $\delta_n$ is arbitrary: indeed,in the example detailed above, whether
$\limsup_n \tau(F(\delta_n),X^n) \leq \tau(F,X)$
depends on how $\delta_n$ compares to $1/n$. 

\begin{proof}
  If $\tau(F,x)>t$ then the compactified trajectory $\Gamma = \overline{x([0,t])}$ is
  entirely contained in $F^c$; by compactness $\Gamma(\delta)$ is also contained in $F^c$
  for $\delta$ small enough. For
  any $t'<t$, by definition of the Skorokhod topology 
  the compactified trajectory $\Gamma_n=\overline{x^n([0,t'])}$ is included
  in $\Gamma(\delta)$ for $n$ large enough, so $\tau(F,x^n) \geq t'$ for $n$ large enough, proving the
  second inequality in \eqref{eq:lemSko1}.

  Similarly, fixing  $\delta$ and $t>\tau(F,x)$,  
  we get that $x(\tau(F,x)) \in F$ so that for $n$ large enough, $x^n(t_n)\in F(\delta)$ for
  some $t_n \leq t$; in other words $\tau(F(\delta),x^n) \leq t$ for $n$ large enough.
  Therefore $\limsup \tau(F(\delta),x^n) \leq \tau$, completing the proof of~\eqref{eq:lemSko1}
  since $t>\tau(F,x)$ is arbitrary. 

  We now prove \eqref{eq:lemSko2}.
  Clearly if  $F\subset G$ then $\tau(F,x)\geq \tau(G,x)$, so $\delta\mapsto \tau(F(\delta),x)$ decreases.
  Let $\delta_n$ be a sequence decreasing to zero: $\tau_n = \tau(F(\delta_n), x)$ is increasing.
  Let $\tau_\infty$ be its limit; since $\tau_n \leq \tau(F,x)$, $\tau_\infty \leq \tau(F,x)$.
  If $\tau_\infty = \infty$ then $\tau_\infty = \tau(F,x)$. If $\tau_\infty$ is finite, 
  for each $n$ one of  $x(\tau_n)$ or $x((\tau_n)_-)$
  is in $F(\delta_n)$: call it $y_n$. By compactness of $\overline{x([0,\tau_\infty])}$
  $y_n$ must converge; its limit is in $\cap_n F(\delta_n) = F$, and is either $x(\tau_\infty)$
  or $x((\tau_\infty)_-)$, so $\tau(F,x)\leq \tau_\infty$ and once more they are equal.

  Suppose $\delta_n$ converges to $0$. Fix a $\delta>0$. For $n$ large enough,
  $\delta_n \leq \delta$ so 
  $\tau(F(\delta_n),x^n) \geq \tau(F(\delta),x^n)$. Taking limits we get
  \[ \liminf \tau(F(\delta_n), x^n) \geq \liminf \tau(F(\delta),x^n) \geq \tau(F(\delta),x),\]
  using \eqref{eq:lemSko1}. Taking $\delta$ to zero and using
  \eqref{eq:lemSko2} yields Equation~\eqref{eq:lemSko3}.  
\end{proof}

The following probabilistic corollary shows that $\delta_n$ may be
chosen to decay slowly enough so that the hitting times converge. 
\begin{Lem}
  \label{lem:choosingNeighborhoods}
  Suppose that $X^n$ converges in distribution to $X$.
  For any closed set $F$, there exists a sequence of radii
  $(\delta_n)_{n\geq 0}$ such that
  \[
    (X^n, \tau(F(\delta_n), X^n)) \xrightarrow[n\to\infty]{(d)} (X, \tau(F,X)).
   \]
 \end{Lem}
 \begin{proof}
   By the Skorokhod representation theorem we may assume without loss of generality
   that $X^n$ converges almost surely to $X$; it is then enough to construct
   $(\delta_n)_{n\geq 0}$ such that $\tau(F(\delta_n),X^n)$ converges in probability to
   $\tau(F,X)$.
   By Lemma~\ref{lem:sko}, we have almost surely, for any sequence $(\delta_n)_{n\geq 0}$, 
   $ \liminf \tau(F(\delta_n),X^n) \geq \tau(F,X)$. To prove the upper bound, we fix $\eps > 0$, and it remains to show that we can construct a sequence $(\delta_n)$ with
   \[\lim_n \prb{ \tau(F(\delta_n,X^n)) > \tau(F,X) + \epsilon} = 0.\]
   Define the events
   \[
     A(n,\delta) \eqdef \left\{ \tau(F(\delta), X^m) \leq \tau(F,X) + \delta, \forall m\geq n\right\},
   \]
   and say that $(n,\delta)$ is good if $\mathbb{P}(A(n,\delta)) \geq 1-\delta$.
   It is easily checked that goodness is doubly monotonous:
   \[ (n'\geq n, \delta' \geq \delta, (n,\delta) \text{ good })
     \implies (n',\delta') \text{ good.}
   \]
   Now, for any fixed $\delta > 0$, the events $\set{ A_{n,\delta} }_{n \geq 1}$ form an increasing sequence,
   and
   \[
     \bigcup_{n\geq 0} A_{n,\delta} = \left\{
       \limsup_{n\geq 0} \tau(F(\delta), X^n) \leq \tau(F,X) + \delta
     \right\}
   \]
   has probability one by \eqref{eq:lemSko1}. As a consequence, for each $\delta > 0$, there is a finite $n(\delta)$ such that $(n(\delta),\delta)$ is good, for instance, 
   $
   n(\delta) = \min \set{ n \geq 1 | (n,\delta) \text{ is good} }
   $;
   and using the monotony of goodness, one can then easily construct a decreasing sequence $(\delta(n))_{n\geq 0}$ that decreases
   to zero and such that $(n,\delta_n)$ is good for each $n \geq 1$.


   Finally, on $A(n,\delta_n)$, $\tau(F(\delta_n,X^n)) \leq \tau(F,X) + \delta_n$, so
   for any $\epsilon>0$, and for $n$ large enough to ensure $\delta_n\leq \epsilon$, 
   \[
     \prb{ \tau(F(\delta_n,X^n)) > \tau(F,X) + \epsilon}
     \leq \prb{ A(n,\delta_n)^c} \leq \delta_n \xrightarrow[n\to\infty]{}0,
   \]
   concluding the proof that $\tau(F(\delta_n),X^n)$ converges to $\tau(F,X)$ in probability. 
 \end{proof}

\bibliographystyle{plain}
\bibliography{biblio}

\end{document}

%% file: macros.tex
\usepackage{lipsum}

\newcommand{\cadlag}{c\`adl\`ag}
\newcommand{\ta}{T}
\newcommand{\PAR}[1]{\left(#1\right)}
\newcommand{\esp}[1]{\mathbb{E}\left[#1\right]}

\DeclareMathOperator{\Vect}{Vect}

\newcommand{\E}{\mathbb{E}}

\newcommand{\N}{\mathbb{N}}
\newcommand{\R}{\mathbb{R}}

\renewcommand{\P}{\mathbb{P}}
\renewcommand{\L}{\mathbb{L}}
\newcommand{\D}{\mathbb{D}}

\newcommand{\calC}{\mathcal{C}}
\newcommand{\calL}{\mathcal{L}}

\newcommand{\calD}{\mathcal{D}}
\newcommand{\calF}{\mathcal{F}}
\newcommand{\calB}{\mathcal{B}}

\newcommand{\calT}{\mathcal{T}}

\newcommand{\calA}{\mathcal{A}}
\newcommand{\calS}{\mathcal{S}}

\newcommand{\eps}{\varepsilon}
\newcommand{\ph}{\varphi}
\renewcommand{\phi}{\varphi}

\newcommand{\Id}{\mathrm{Id}}

\newcommand{\e}{{\rm e}}
\renewcommand{\d}{ {\rm d}}

\newcommand{\eqdef}{ \dps \mathop{=}^{{\rm def}} }
\newcommand{\one}{ {\rm l} \hspace{-.7 mm} {\rm l}}
\newcommand{\dps}{\displaystyle}

\newcommand{\abs}[1]{\left | #1\right |}
\newcommand{\set}[1]{\left\{#1\right\}}
\newcommand{\p}[1]{ \left(#1\right) }
\renewcommand{\b}[1]{\left [ #1\right ]}

\newcommand{\norm}[1]{\left\Vert#1\right\Vert}
\newcommand{\bracket}[1]{\left \langle #1\right \rangle}

\newcommand{\cco}{\llbracket}
\newcommand{\ccf}{\rrbracket}
\newcommand{\na}{\nabla}
 
\theoremstyle{plain}
\newtheorem{The}{Theorem}[section]
\newtheorem{Lem}[The]{Lemma}

\newtheorem{Def}[The]{Definition}
\newtheorem{Ass}[The]{Assumption}

\numberwithin{equation}{section}

\theoremstyle{definition}
\newtheorem{Rem}[The]{Remark}